\documentclass[a4paper]{amsart}
\usepackage{amsfonts,amsmath,amsthm,amssymb,color}
\usepackage[all]{xy}

\usepackage[pagebackref, colorlinks=true, linkcolor=blue, citecolor=blue]{hyperref}

\numberwithin{equation}{section}

\begin{document}

\newtheorem{theorem}{Theorem}[section]
\newtheorem{lemma}[theorem]{Lemma}
\newtheorem{proposition}[theorem]{Proposition}
\newtheorem{corollary}[theorem]{Corollary}

\theoremstyle{definition}
\newtheorem{definition}[theorem]{Definition}
\newtheorem{example}[theorem]{Example}

\theoremstyle{remark}
\newtheorem{remark}[theorem]{Remark}
\newtheorem*{ack}{Acknowledgments}

\newenvironment{magarray}[1]
{\renewcommand\arraystretch{#1}}
{\renewcommand\arraystretch{1}}

\newcommand{\mapor}[1]{{\stackrel{#1}{\longrightarrow}}}
\newcommand{\mapin}[1]{{\stackrel{#1}{\hookrightarrow}}}
\newcommand{\mapver}[1]{\Big\downarrow\vcenter{\rlap{$\scriptstyle#1$}}}

\newcommand{\liminv}{\smash{\mathop{\lim}\limits_{\leftarrow}\,}}

\newcommand{\Set}{\mathbf{Set}}
\newcommand{\Art}{\mathbf{Art}}
\newcommand{\Grpd}{\mathbf{Grpd}}

\renewcommand{\bar}{\overline}
\newcommand{\de}{\partial}
\newcommand{\debar}{{\overline{\partial}}}
\newcommand{\per}{\!\cdot\!}

\newcommand{\Oh}{\mathcal{O}}
\newcommand{\sA}{\mathcal{A}}
\newcommand{\sB}{\mathcal{B}}
\newcommand{\sC}{\mathcal{C}}
\newcommand{\sD}{\mathcal{D}}
\newcommand{\sE}{\mathcal{E}}
\newcommand{\sF}{\mathcal{F}}
\newcommand{\sG}{\mathcal{G}}
\newcommand{\sH}{\mathcal{H}}
\newcommand{\sI}{\mathcal{I}}
\newcommand{\sJ}{\mathcal{J}}
\newcommand{\sK}{\mathcal{K}}
\newcommand{\sL}{\mathcal{L}}
\newcommand{\sM}{\mathcal{M}}
\newcommand{\sO}{\mathcal{O}}
\newcommand{\sP}{\mathcal{P}}
\newcommand{\sU}{\mathcal{U}}
\newcommand{\sV}{\mathcal{V}}
\newcommand{\sX}{\mathcal{X}}
\newcommand{\sY}{\mathcal{Y}}
\newcommand{\sN}{\mathcal{N}}
\newcommand{\sZ}{\mathcal{Z}}

\newcommand{\Aut}{\operatorname{Aut}}
\newcommand{\Mor}{\operatorname{Mor}}
\newcommand{\Def}{\operatorname{Def}}
\newcommand{\Hom}{\operatorname{Hom}}
\newcommand{\Hilb}{\operatorname{Hilb}}
\newcommand{\HOM}{\operatorname{\mathcal H}\!\!om}
\newcommand{\DER}{\operatorname{\mathcal D}\!er}
\newcommand{\Spec}{\operatorname{Spec}}
\newcommand{\Der}{\operatorname{Der}}
\newcommand{\End}{\operatorname{End}}
\newcommand{\END}{\operatorname{\mathcal E}\!\!nd}
\newcommand{\Image}{\operatorname{Im}}
\newcommand{\coker}{\operatorname{coker}}
\newcommand{\tot}{\operatorname{tot}}
\newcommand{\Diff}{\operatorname{Diff}}
\newcommand{\Del}{\operatorname{Del}}
\newcommand{\Tot}{\operatorname{Tot}}
\newcommand{\MC}{\operatorname{MC}}
\newcommand{\Coder}{\operatorname{Coder}}
\newcommand{\id}{\operatorname{id}}
\newcommand{\ad}{\operatorname{ad}}

\renewcommand{\Hat}[1]{\widehat{#1}}
\newcommand{\dual}{^{\vee}}
\newcommand{\desude}[2]{\dfrac{\de #1}{\de #2}}
\newcommand{\A}{\mathbb{A}}
\newcommand{\N}{\mathbb{N}}
\newcommand{\R}{\mathbb{R}}
\newcommand{\Z}{\mathbb{Z}}
\renewcommand{\H}{\mathbb{H}}
\renewcommand{\L}{\mathbb{L}}
\newcommand{\proj}{\mathbb{P}}
\newcommand{\K}{\mathbb{K}\,}
\newcommand\C{\mathbb{C}}
\newcommand\T{\mathbb{T}}
\newcommand{\contr}{{\mspace{1mu}\lrcorner\mspace{1.5mu}}}

\newcommand{\bi}{\boldsymbol{i}}
\newcommand{\bl}{\boldsymbol{l}}

%\newcommand\é{\'e}
%\newcommand\è{\`e}
%\newcommand\à{\`a}
%\newcommand\ì{\`i}
%\newcommand\ù{\`u}
%\newcommand\ò{\`o }
%\newcommand\vr{``}
%%%%%%%%%%%%%%%%%%%%%%%%%%%%%%%%%%%%%%%%%%%%%%%%%%%%%%%%%%%%%%%%%%%%5555

\title{On coisotropic deformations of  holomorphic submanifolds}
\author{Ruggero Bandiera}
\address{\newline
Universit\`a degli studi di Roma La Sapienza,\hfill\newline
Dipartimento di Matematica \lq\lq Guido
Castelnuovo\rq\rq,\hfill\newline
P.le Aldo Moro 5,
I-00185 Roma, Italy.}
\email{bandiera@mat.uniroma1.it}

\author{Marco Manetti}
\email{manetti@mat.uniroma1.it}
\urladdr{www.mat.uniroma1.it/people/manetti/}

\date{May 3, 2013}

\begin{abstract}  We describe the differential graded Lie algebras governing Poisson 
deformations of a holomorphic Poisson manifold and coisotropic embedded deformations 
of a coisotropic holomorphic submanifold. In both cases, under some mild additional assumption, we show that the infinitesimal first order deformations induced by the anchor map are unobstructed.
Applications include the analog of Kodaira stability 
theorem for coisotropic deformation and a generalization of McLean-Voisin's theorem about 
the local moduli space of lagrangian submanifold. Finally it is shown that our construction 
is homotopy equivalent to the  homotopy Lie algebroid, in the cases where this is defined.
\end{abstract}


\maketitle

\section{Introduction}

The classical notion of coisotropic submanifold of a symplectic manifold extends immediately to the  setup of Poisson geometry. 
More precisely,  a closed submanifold $Z\subset X$ is called coisotropic if the ideal of $Z$ is stable under the Poisson bracket.
In recent years coisotropic submanifolds, and their cohomology, have received a lot of attention in view of their importance in mathematics and physics (see e.g. \cite{CF04}).
The definition of coisotropic submanifolds extends literally to the complex holomorphic case and more generally to every  algebraic Poisson variety over a field. 

The goal of this paper is to study deformation theory of holomorphic coisotropic submanifold. The existing approach to deformations of differentiable coisotropic submanifold, based on the notion of  ``homotopy Lie algebroid'' \cite{CS08,OP05},
does not extend immediately to the holomorphic case since it relies on the identification of a 
neigbourhood with the total space of the normal bundle. 

Here we follows a different approach, based on  descent theory for the  Deligne groupoid associated 
to a sheaf of differential graded Lie algebras. Our approach is, in great part, purely algebraic and 
then most of the  results of this paper also works for algebraic manifolds over a field of characteristic 0.

Given a holomorphic Poisson bivector $\pi$ on a holomorphic manifold $X$, the 
Lichnerowicz-Poisson differential $d_{\pi}=[\pi,\cdot]_{SN}$ is a square zero operator on the sheaf of holomorphic polyvector fields on $X$. Then a closed submanifold $Z\subset X$ is coisotropic  if and only if $d_{\pi}$ factors to a differential on the exterior algebra of the normal sheaf $\sN_{Z|X}$ of $Z$ in $X$.

In the first part of the paper we show that the functor $\Hilb_{Z|X}^{co}$ of infinitesimal embedded coisotropic deformations of $Z$ in $X$ is governed by a differential graded Lie algebra $K$, explicitly described, having cohomology isomorphic to the hypercohomology of the complex of sheaves over $Z$:
\[ {\bigwedge}^{\ge 1}\sN_{Z|X}:\qquad
\sN_{Z|X}\xrightarrow{d_{\pi}}{\bigwedge}^2 \sN_{Z|X}\xrightarrow{d_{\pi}}\cdots\]
where ${\bigwedge}^i \sN_{Z|X}$ is considered in degree $i$.
In particular, the space of  infinitesimal first order deformations is isomorphic to 
\[ H^1(K)=\H^1({\bigwedge}^{\ge 1}\sN_{Z|X})=\ker(d_{\pi}\colon H^0(Z,\sN_{Z|X})\to 
H^0(Z,{\bigwedge}^2\sN_{Z|X}))\]
and there exists a complete obstruction theory with values in 
$H^2(K)=\H^2({\bigwedge}^{\ge 1}\sN_{Z|X})$. As a byproduct we also obtained the explicit description of two differential graded Lie algebra governing respectively the deformations of the pair $(X,\pi)$, i.e., the deformations of $X$ as a Poisson manifold and the deformations of the triple 
$(X,Z,\pi)$.

In the second part we consider the effect of the anchor 
map $\pi^{\#}\colon \Omega^*_X\to \bigwedge^*\Theta_X$ on 
deformations of coisotropic submanifolds: by definition  
$\pi^{\#}$ is the unique morphism of sheaves of (graded) 
$\Oh_{X}$-algebras 
such that $ \pi^{\#}(df)=d_{\pi}(f)$ for every $f\in \Oh_X$.
Very recently,  N. Hitchin \cite{hitchin} has proved that, 
if $X$ is a compact K\"{a}hler Poisson manifold,
then  every element in the image of 
$\pi^{\#}\colon H^1(X,\Omega_X^1)\to  H^1(X,\Theta_X)$ is the Kodaira-Spencer class of a 
deformation of the pair $(X,\pi)$ over a germ of smooth curve. 
Here we prove a similar statement for embedded coisotropic deformations: given a coisotropic submanifold
$Z\subset X$, the anchor map factors to a morphism of sheaves of (graded) $\Oh_Z$-algebras
$\pi^{\#}\colon \Omega_Z^*\to \bigwedge^*\sN_{Z|X}$. If the Hodge to de Rham spectral sequence of $Z$ degenerate at $E_1$, then 
every element in the 
image of $\pi^{\#}\colon H^0(Z,\Omega_Z^1)\to H^0(Z,\sN_{Z|X})$ is 
the Kodaira-Spencer class of a coisotropic embedded
deformation of $Z$ in  $(X,\pi)$ over a  germ of smooth curve.
Such result applied to a compact K\"{a}hler 
Lagrangian submanifold $Z$ of a holomorphic symplectic manifold $X$ 
shows that every small deformation in $X$ of $Z$ is Lagrangian 
and the Hilbert scheme of $X$ is smooth at 
$Z$; when $X$ is compact K\"{a}hler we recover in this way a 
classical result by  Voisin and McLean~\cite{voisin,mclean}.
 
The underlying idea of proof, borrowed from \cite{FMpoisson}, is to show that the 
anchor map is equivalent, in the homotopy category of differential graded Lie algebras, 
to a morphism $\pi^{\#}\colon J\to K$, where the cohomology of $J$ is isomorphic to  
the hypercohomology of 
the complex of sheaves on $Z$ 
\[ \Omega_Z^{\ge 1}:\qquad
\Omega_Z^1\xrightarrow{d}\Omega_Z^2\xrightarrow{d}\cdots \]
where $\Omega^i_Z $ is considered in degree $i$. The formality criterion of \cite{FMpoisson} 
applies and, whenever the Hodge to de Rham spectral sequence of $Z$ degenerates at $E_1$, the differential graded Lie algebra $J$ is homotopy abelian, hence governing unobstructed deformations.

In the last part of the paper we compare our construction with the homotopy Lie algebroid of 
Oh, Park, Cattaneo and Felder \cite{CF07,OP05} in the situation where the latter may be defined. 
As expected,  the two constructions are homotopy equivalent and then, according to the general principles 
of derived deformation theory, they define the 
same deformation problem, both classical and extended.
It is worth to mention here that our construction provides the basic data 
for the local study of the extended moduli space of 
coisotropic submanifolds, extending  the  Lagrangian case carried out by Merkulov 
\cite{merkulov}.

\begin{ack} The authors thank Domenico Fiorenza and Rita Pardini for useful discussions on the subject of this paper.
\end{ack}

\bigskip
\section{Review of Deligne groupoids and totalization}
\label{sec.Deligne and totalization}

Denote by $\Set$  the category of sets (in a fixed universe) and
by $\mathbf{Grpd}$ the category of  groupoids; we shall consider, in the obvious way,
$\Set$ as a full subcategory of $\mathbf{Grpd}$.

Given a field $\K$ we shall denote by  $\Art_{\K}$  the category of local Artin
$\K$-algebras with residue field $\K$. Unless otherwise specified,
for every  object   $A\in \mathbf{Art}_{\K}$, we denote by
$\mathfrak{m}_A$ its maximal ideal.

In order to simplify the terminology,
by a formal pointed groupoid  we shall mean a covariant functor
$\sF\colon\Art_{\K} \to \mathbf{Grpd}$ such that $\sF(\K)=*$ is the one-point set. Similarly a formal 
pointed set is a functor $F\colon\Art_{\K}\to \mathbf{Set}$ such that $F(\K)=*$, also called a functor of Artin rings.
A morphism of formal pointed groupoid $\eta\colon \sF\to \sG$ is called an equivalence if $\sF(A)\to \sG(A)$ is an equivalence of groupoids for every $A\in \Art_{\K}$.

It is a nowadays standard to consider every deformation problem over a field of characteristic 0
as controlled by a differential graded Lie algebra (DGLA);
we refer to the existing literature and in particular to
\cite{GoMil1,K,ManettiSeattle} for the definition and main properties
of differential graded Lie algebras, $L_{\infty}$-algebras, Maurer-Cartan equation and gauge action.
Later in this paper we also need to work with   
$L_{\infty}[1]$-algebras, i.e., desuspension of $L_{\infty}$-algebras: we refer to    
\cite{FZ12} for a nice and clear introduction to these structures.

The Deligne groupoid of a differential graded Lie algebra $L$ over a field $\K$ of characteristic 0 is the formal pointed groupoid
\[ \Del_{L}\colon \Art_{\K}\to \mathbf{Grpd}\]
defined in the following way \cite{DtM,GoMil1}:
given $A\in \Art_{\K}$ the objects of $\Del_{L}(A)$ are the solutions of the Maurer-Cartan equation in $L\otimes\mathfrak{m}_A$:
\[ \text{Objects}(\Del_{L}(A))=\left\{x\in L^1\otimes\mathfrak{m}_A\;\middle|\; 
dx+\dfrac{1}{2}[x,x]=0\right\}.\]
Given two objects $x,y$ of $\Del_{L}(A)$, the morphisms between them are
\[ \Mor_{\Del_{L}(A)}(x,y)=\{e^a\in \exp(L^0\otimes \mathfrak{m}_A)\mid e^a\ast x=y\},\]
where $\ast$ is the gauge action:
\[
e^a \ast x:=x+\sum_{n\geq 0} \frac{ [a,-]^n}{(n+1)!}([a,x]-da)\;.
\]

The deformation functor associated to a differential graded Lie algebra $L$ is the $\pi_0$ of the Deligne groupoid:
\[ \Def_{L}\colon \Art_{\K}\to \mathbf{Set},\qquad
\Def_L(A)=\pi_0(\Del_{L}(A)).
\]
The tangent space $T^1\Def_{L}$ of the functor $\Def_{L}$ is isomorphic to the cohomology group $H^1(L)$.
The homotopy invariance of $\Del$ and $\Def$ is summarized by the following result.

\begin{theorem}\label{thm.homotopyinvariance}
Let $L\to M$ be a quasi-isomorphism of differential graded Lie algebras. Then:
\begin{enumerate}

\item the induced natural transformation $\Def_{L}\to \Def_M$ is an isomorphism;

\item if $L$ and $M$ are positively graded, i.e., $L^i=M^i=0$ for every $i<0$, then
 the induced natural transformation $\Del_{L}\to \Del_M$ is an equivalence.
\end{enumerate}
\end{theorem}

\begin{proof} The second item is one of the main results of \cite{GoMil1}.
The first item is proved in \cite{K} via homotopy classification of $L_{\infty}$-algebras,
in \cite{Kont94,ManettiDGLA} via reduced Deligne groupoid and in \cite{EDF} via extended deformation functors.
\end{proof}

Recall that a differential graded Lie algebra is called homotopy abelian if it is
quasi-isomorphic to an abelian DGLA. Equivalently,  a DGLA $L$ is homotopy abelian if
there exists a diagram of differential graded Lie algebras
\[ L\xleftarrow{\;f\;}M\xrightarrow{\;g\;} H\]
where $f$ and $g$ are quasi-isomorphisms and $H$ is abelian, i.e., the bracket of $H$ is trivial.

The Theorem~\ref{thm.homotopyinvariance}
immediately implies that the functor $\Def_L$ is unobstructed whenever $L$ is homotopy abelian.

\begin{proposition}\label{prop.injective in cohomology}
 Let $f\colon L\to M$ be a morphism of differential graded Lie algebras:
\begin{enumerate}

\item If $f$ is injective in cohomology and
$M$ is  homotopy abelian,
then also $L$ is homotopy abelian;

\item If $f$ is surjective in cohomology and
$L$ is  homotopy abelian,
then also $M$ is homotopy abelian.
\end{enumerate}
\end{proposition}

\begin{proof} This is proved in  \cite[Proposition~4.11]{KKP} by using the reduction
to $L_{\infty}$ minimal models.
For a different proof of the first item see  also \cite[Lemma 6.1]{FMpoisson}.
\end{proof}

An useful tool in the framework of deformation theory via DGLA is the homotopy fiber construction:
given a differential graded Lie algebra $M$ we shall denote $M[t,dt]=\K[t,dt]\otimes M$, where
$\K[t,dt]=\K[t]\oplus \K[t]dt$
is the de Rham differential graded algebra of polynomial forms on the affine line.
For every $s\in \K$, the evaluation at $s$ gives a morphism of DGLAs:
\[ e_s\colon M[t,dt]\to M,\qquad t\mapsto s,\; dt\mapsto 0.\]
It is useful to consider also the morphism of graded vector spaces
\[ \int_0^1\colon M[t,dt]\to M[-1], \qquad   \int_0^1 t^a\otimes m=0,\quad \int_0^1 t^adt\otimes
m=\frac{m}{a+1}.\]

Given a morphism $\chi\colon L\to M$ of differential graded Lie algebras, its \emph{homotopy fiber} is defined as the DGLA
\[ K(\chi)=\{(l,m(t))\in L\times M[t,dt]\mid e_0(m(t))=0,\; e_1(m(t))=\chi(l)\}.\]
Notice that the projection $K(\chi)\to L$ is a morphism of DGLA.

A DG-vector space is a $\Z$-graded vector space equipped with a differential of degree $+1$.
For an integer $p$, the $p$-th suspension of a DG-vector space $(V,d_V)$
is equal to $(V[-p],d_{V[-p]})$, where
\[ V[-p]^i=V^{i-p},\qquad d_{V[-p]}=(-1)^pd_V\;.\]

\begin{lemma} In the notation above, if $\chi\colon L\to M$ is an injective morphism of DGLA, then the map
\[ \int_0^1\colon K(\chi)\to \coker(\chi)[-1],\quad \int_0^1(l,m(t))=\int_0^1 m(t)\pmod{\chi(L)[-1]},\]
is a quasi-isomorphism of complexes.
\end{lemma}

\begin{proof} Straightforward, cf.
\cite{FMcone,ManettiSemireg}.\end{proof}

Let $\mathbf{\Delta}_{\operatorname{mon}}$ be  the category whose
objects are the finite ordinal sets $[n]\!=\!\{0,1,\ldots,n\}$,
$n=0,1,\ldots$, and whose morphisms are order-preserving injective
maps among them. Every morphism in
$\mathbf{\Delta}_{\operatorname{mon}}$, different from the
identity, is a finite  composition of face maps:
\[
\partial_{k}\colon [i-1]\to [i],
\qquad \partial_{k}(p)=\begin{cases}p&\text{ if }p<k\\
p+1&\text{ if }k\le p\end{cases},\qquad k=0,\dots,i.
\]
The relations about compositions of them are generated by
\[ \partial_{l}\partial_{k}=
\partial_{k+1}\partial_{l}\, ,\qquad\text{for every }l\leq k.\]

Following the terminology of \cite{weibel}, a semicosimplicial object in
a category $\mathbf{C}$ is a  covariant functor $A_{\bullet}\colon
\mathbf{\Delta}_{\operatorname{mon}}\to \mathbf{C}$.
Equivalently,
a semicosimplicial  object  is a diagram in
$\mathbf{C}$:
 \[A_{\bullet}\colon\qquad
\xymatrix{ {A_0}
\ar@<2pt>[r]\ar@<-2pt>[r] & { A_1}
      \ar@<4pt>[r] \ar[r] \ar@<-4pt>[r] & { A_2}
\ar@<6pt>[r] \ar@<2pt>[r] \ar@<-2pt>[r] \ar@<-6pt>[r]&
\cdots ,}
\]
where each $A_i$ is in $\mathbf{C}$, and, for each $i>0$,
there are $i+1$ morphisms
\[
\partial_{k}\colon {A}_{i-1}\to {A}_{i},
\qquad k=0,\dots,i,
\]
such that $\partial_{l}
\partial_{k}=\partial_{k+1} \partial_{l}$, for any $l\leq k$.\par

In this paper we need to consider semicosimplicial groupoids
and semicosimplicial differential graded Lie algebras. In both cases we can perform the totalization construction; let's first consider the case of groupoids.

Given a semicosimplicial groupoid
 \[
\sG_{\bullet}:\qquad \xymatrix{ {{\sG}_0}
\ar@<2pt>[r]\ar@<-2pt>[r] & { {\sG}_1}
      \ar@<4pt>[r] \ar[r] \ar@<-4pt>[r] & { {\sG}_2}
\ar@<6pt>[r] \ar@<2pt>[r] \ar@<-2pt>[r] \ar@<-6pt>[r]&
\cdots}
\]
the groupoid
$\operatorname{Tot}(\sG_{\bullet})$, also called \emph{groupoid of descent data}, is
defined in the following way \cite{hinichdescent,semireg2011}:
\begin{enumerate}

\item The objects of $\operatorname{Tot}(\sG_{\bullet})$ are the
pairs $(l,m)$ with $l$ an object in $\sG_0$ and $m$ a morphism in $\sG_1$ between $
\partial_{0}l$ and $\partial_{1}l$; we require that the three images of $m$ via the maps $\partial_{i}$ are the edges of a 2-simplex in the nerve of $\sG_2$, i.e.,
\[   (\partial_{0}m)(\partial_{1}m)^{-1}(\partial_{2}m)=1 \text{ in }
\Mor_{\sG_2}(\partial_{2}\partial_{0}l,\partial_{2}\partial_{0}l).\]

\item The morphisms between $(l_0,m_0)$ and $(l_1,m_1)$ are the morphisms
$a\in \Mor_{\sG_0}(l_0,l_1)$ making the diagram
\[
\xymatrix{
\partial_{0}l_0\ar[r]^{m_0}\ar[d]_{\partial_{0}a}&\partial_{1}l_0\ar[d]^{{\partial_{1}a}}\\
\partial_{0}l_1\ar[r]^{m_1}&\partial_{1}l_1
}
\]
commutative in $\sG_1$.
\end{enumerate}

The totalization of semicosimplicial groupoids is functorial and commutes with equivalences: more precisely, if
$\gamma\colon \sF_{\bullet}\to \sG_{\bullet}$ is a morphism of semicosimplicial groupoids, then
$\Tot(\gamma)\colon \Tot(\sF_\bullet)\to \Tot(\sG_\bullet)$ is a morphism of groupoids and, if every
$\gamma_n\colon \sF_n\to \sG_n$ is an equivalence of groupoids, then also $\Tot(\gamma)$ is an equivalence of groupoids.

Let
\[V_{\bullet}:\quad
\xymatrix{ {V_0}
\ar@<2pt>[r]\ar@<-2pt>[r] & { V_1}
      \ar@<4pt>[r] \ar[r] \ar@<-4pt>[r] & { V_2}
\ar@<6pt>[r] \ar@<2pt>[r] \ar@<-2pt>[r] \ar@<-6pt>[r]& \cdots ,}\]
be a semicosimplicial DG-vector space, with face operators $\de_i\colon V_n\to V_{n+1}$. Then
the graded vector space $C(V_\bullet)=\prod_{n\ge 0}V_n[-n]$ carries the two
differentials
\[ d=\prod_{n}d_{V_n[-n]}=\prod_{n}(-1)^nd_{V_n}\quad\text{and}\quad
\partial=\sum_{i}(-1)^i\partial_i\;.\]
More explicitly, if
$v\in V^i_n$, then $d(v)=(-1)^nd_{V_n}(v)\in V^{i+1}_n$ and
$\partial(v)=\partial_0(v)-\partial_1(v)+\cdots+(-1)^{n+1}
\partial_{n+1}(v)\in V_{n+1}^i$.
Since $d^2=\partial^2=d\partial+\partial d=0$ we may define the
cochain complex of $V_{\bullet}$ as  the differential graded vector space
$C(V_\bullet)$ equipped with the differential $d+\partial$.

Let $\Omega_n$ be the polynomial de Rham algebra of the standard $n$-dimensional simplex.
In other words, $\Omega_n$ is the polynomial DG-algebra generated by $t_0,\ldots,t_n$ of degree zero and $dt_0,\ldots,dt_n$ of degree one subject to the relations $t_0+\cdots+t_n=1$ and $dt_0+\cdots+dt_n=0$.
The entire collection $\{\Omega_n\}$, $n\ge 0$, has a natural structure of simplicial DG-algebras, where
the face map $\de_i\colon [n-1]\to [n]$ induces the morphism of DG-algebras
\[ \de_i^*\colon \Omega_n\to \Omega_{n-1},\qquad \de_i^*t_k=
\begin{cases}t_k&\text{ for }k<i\\
0&\text{ for }k=i\\
t_{k-1}&\text{ for }k>i\;,\end{cases}\qquad \de_i^*(dt_k)=d(\de_i^*t_k)\,.\]

\begin{definition}
The (Thom-Whitney-Sullivan)
totalization of a semicosimplicial DG-vector space
\[V_{\bullet}:\qquad
\xymatrix{ {V_0}
\ar@<2pt>[r]\ar@<-2pt>[r] & { V_1}
      \ar@<4pt>[r] \ar[r] \ar@<-4pt>[r] & { V_2}
\ar@<6pt>[r] \ar@<2pt>[r] \ar@<-2pt>[r] \ar@<-6pt>[r]& \cdots ,}\]
is defined as
\[ \Tot(V_\bullet)=\left\{(x_n)\in \prod_{n\ge 0} \Omega_n\otimes V_n\;\middle|\;
(\partial_k^*\otimes Id)x_n=
(Id\otimes \partial_{k})x_{n-1}\; \text{ for every }\; 0\le k\le
n\right \}.\]
\end{definition}

The functor $\Tot$ is exact: given a  sequence $0\to V_\bullet\to W_\bullet\to U_\bullet\to 0$
of semicosimplicial DG-vector spaces such that
$0\to V_n\to W_n\to U_n\to 0$ is exact for every $n$, then also
$0\to \Tot(V_\bullet)\to \Tot(W_\bullet)\to \Tot(U_\bullet)\to 0$ is exact.
This follows immediately from the definition and the simplicial contractibility  of the simplicial 
DG-algebra $\Omega_\bullet$ \cite[Prop. 1.1]{BG76}.

By Stokes formula, the \emph{Whitney integration  map}
\[ I\colon \Tot(V_\bullet)\to C(V_\bullet)\,,\]
defined componentwise as
\[ \Tot(V_\bullet)^{p}\xrightarrow{\;\text{inclusion}\;}\prod_{n\ge 0}({\bigoplus}_i \Omega_n^{n-i}\otimes V_n^{p-n+i})\xrightarrow{\;\prod_n \int_{\Delta^n}\otimes Id_{V_n}\;}\prod_n V_n^{p-n}=C(V_\bullet)^p\,,\]
is a surjective morphism of DG-vector spaces and, by a theorem of Whitney, Thom, Sullivan and Dupont (see e.g. \cite{getzler04,navarro} for a proof), it is also a quasi-isomorphism.

\begin{example}\label{ex.coomologiadelTotversusHypercohomology}
Let $\sF^*$ be a bounded below complex of coherent sheaves on a complex manifold $X$. Given a open Stein covering $\sU=\{U_i\}$ of $X$
we can consider the semicosimplicial DG-vector space
\[ \sF^*(\sU)_{\bullet}:\qquad
   \xymatrix{\displaystyle\prod_{i}^{\phantom{i}}\sF^*(U_{i})\ar@<0.65ex>[r]\ar@<-0.65ex>[r]&
\displaystyle\prod_{i,j}^{\phantom{i}}\sF^*(U_{ij})\ar@<1.00ex>[r]\ar[r]\ar@<-1.00ex>[r]&
\displaystyle\prod_{i,j,k}^{\phantom{i}}\sF^*(U_{ijk})\cdots}\]
and then for every integer $p$ we have an isomorphism
\[ H^p(\Tot(\sF^*(\sU)_{\bullet}))\simeq H^p(C(\sF^*(\sU)_{\bullet}))=\H^p(X,\sF^*),\]
where $\H^*$ denotes the hypercohomology groups.
\end{example}

When $A_{\bullet}$ if a semicosimplicial algebra (either associative or Lie), then $\Tot(A_{\bullet})$
inherits a natural structure of algebra and, via the quasi-isomorphism $I$, this structure
induces in the cohomology of $C(A_\bullet)$ not only the cup products, but also the higher Massey products.
Morally,  the totalization is the smallest
natural differential graded multiplicative structure giving the correct cohomology of a
semicosimplicial algebra \cite[pag. 300]{sullivan}.

We are now ready to recall Hinich's theorem on descent of Deligne groupoids.

\begin{theorem}[Descent of Deligne groupoids, \cite{hinichdescent}]
\label{thm.hinichdescent}
Let
\[L_{\bullet}:\qquad
\xymatrix{ {L_0}
\ar@<2pt>[r]\ar@<-2pt>[r] & { L_1}
      \ar@<4pt>[r] \ar[r] \ar@<-4pt>[r] & { L_2}
\ar@<6pt>[r] \ar@<2pt>[r] \ar@<-2pt>[r] \ar@<-6pt>[r]& \cdots ,}\]
be a semicosimplicial differential graded Lie algebra concentrated in positive degrees, i.e., $L_i^j=0$ for every $i$ and every $j<0$.
Then there exists a natural equivalence of formal pointed groupoids
\[ \Del_{\Tot(L_\bullet)}\to \Tot(\Del_{L_\bullet})\;,\]
where $\Del_{L_\bullet}$ is the semicosimplicial formal groupoid
\[\Del_{L_{\bullet}}:\qquad
\xymatrix{ \Del_{L_0}
\ar@<2pt>[r]\ar@<-2pt>[r] & \Del_{L_1}
      \ar@<4pt>[r] \ar[r] \ar@<-4pt>[r] & \Del_{L_2}
\ar@<6pt>[r] \ar@<2pt>[r] \ar@<-2pt>[r] \ar@<-6pt>[r]& \cdots\quad .}\]
\end{theorem}

\begin{example}\label{ex.deligneperfibraootopica}
The simplest non trivial application of Theorem~\ref{thm.hinichdescent} describes  a ``small'' model for the Deligne groupoid of the homotopy fiber of a morphism $\chi\colon L\to M$ of positively graded differential graded Lie algebras. In fact, the morphism $\chi$ gives the semicosimplicial DGLA
\[ \chi_{\bullet}:\qquad
\xymatrix{ {L}
\ar@<2pt>[rr]^{\de_0=\chi}\ar@<-2pt>[rr]_{\de_1=0} && { M}
      \ar@<4pt>[r] \ar[r] \ar@<-4pt>[r] & {0}
\ar@<6pt>[r] \ar@<2pt>[r] \ar@<-2pt>[r] \ar@<-6pt>[r]& \cdots ,}\]
and, since $\Omega_1\simeq \K[t,dt]$ we have $K(\chi)\simeq \Tot(\chi_\bullet)$; therefore there exists a natural equivalence of formal groupoids
\[ \Del_{K(\chi)}\to \Tot\left(\xymatrix{ {\Del_L}
\ar@<2pt>[r]^{\chi}\ar@<-2pt>[r]_{0} & { \Del_M} }\right)\;.\]
Here, for every $A\in \Art_{\K}$, the objects of $ \Tot\left(\xymatrix{ {\Del_L}
\ar@<2pt>[r]^{\chi}\ar@<-2pt>[r]_{0} & { \Del_M} }\right)(A)$ are the pairs $(x,e^a)$, where
$x$ is a solution of the Maurer-Cartan equation in $L\otimes\mathfrak{m}_A$ and $e^a\in \exp(M^0\otimes\mathfrak{m}_A)$ satisfies the equation $e^a\ast \chi(x)=0$.

A morphism between two objects $(x,e^a)$ and $(y,e^b)$ is an element $e^{\alpha}\in
\exp(L^0\otimes\mathfrak{m}_A)$ such that $e^\alpha\ast x=y$ and
$e^b=e^ae^{-\chi(\alpha)}$.

If $L$ is a differential graded Lie subalgebra of $M$ and $\chi$ is the inclusion map,
then the objects of the total groupoid are  in bijection with the
elements $e^a\in \exp(M^0\otimes\mathfrak{m}_A)$ such that
$e^{-a}\ast 0\in L\otimes\mathfrak{m}_A$; moreover, in this particular case the natural transformation
\[ \Tot\left(\xymatrix{ {\Del_L}
\ar@<2pt>[r]^{\chi}\ar@<-2pt>[r]_{0} & { \Del_M} }\right)\to
\pi_0\left(\Tot\left(\xymatrix{ {\Del_L}
\ar@<2pt>[r]^{\chi}\ar@<-2pt>[r]_{0} & { \Del_M} }\right)\right)\]
is an equivalence of formal groupoids.
\end{example}

We conclude this section with some  remarks that will be useful in this paper.

\begin{remark}
Let $f\colon L_\bullet\to M_\bullet$ be a morphism of semicosimplicial differential graded Lie algebras: this is given by a sequence of morphisms of DGLA $f_n\colon L_n\to M_n$ commuting with face operators.
Taking homotopy fibers we get a semicosimplicial DGLA $K(f)_\bullet$,  where $K(f)_n=K(f_n)$ is
the homotopy fiber of $f_n$.\par

It is easy to see that the homotopy fiber of the morphism $f\colon \Tot(L_\bullet)\to \Tot(M_\bullet)$
is naturally isomorphic to $\Tot(K(f)_\bullet)$; we refer to \cite{donarendiconti} for a detailed proof.
\end{remark}

\begin{remark}\label{rem.equalizzatorecosimpliciale}
Let $L_\bullet$ be a semicosimplicial DGLA and denote by $H=\{x\in L_0\mid \de_0x=\de_1x\}$ the equalizer of
$\de_0,\de_1\colon L_0\to L_1$. Then the map
\[ e\colon H\to \Tot(L_\bullet),\qquad
e(x)=(1\otimes x, 1\otimes \de_0 x,1\otimes \de_0^2 x,1\otimes \de_0^3 x,\ldots),\]
is a well defined morphism of differential graded Lie algebras. Moreover, the composition of $e$ with the quasi-isomorphism $I\colon \Tot(L_\bullet)\to C(L_\bullet)$ is the natural inclusion
$i\colon H\hookrightarrow C(L_\bullet)$, cf. \cite{BGNT,semireg2011}. In particular $e$ is a quasi-isomorphism if and only if $i$ is a quasi-isomorphism.
\end{remark}

\bigskip
\section{Differential Gerstenhaber algebras in Poisson geometry}

Let $X$ be a complex manifold and let $\Oh_{X}$ denote  the sheaf of holomorphic functions on $X$ and with $\Theta_{X}$ the holomorphic tangent sheaf, also denote by $\bigwedge^{\ast}\Theta_{X}=\bigoplus_{i\geq0}\bigwedge_{\Oh_{X}}^{i}\Theta_{X}[-i]$ the sheaf of holomorphic polyvector fields and with $(\Omega^{\ast}_{X},\de)$ the sheaf of holomorphic differential forms on $X$, the last one with the usual structure of sheaf of differential graded algebras (DGAs for short). Following the notation of \cite{FMpoisson},
given a polyvector field  $\eta\in\bigwedge^{i}\Theta_{X}(U)$ we denote by
\[\bi_{\eta}\colon\Omega^{\ast}_{X}(U)\rightarrow\Omega^{\ast-i}_{X}(U),\qquad \bi_{\eta}(\alpha)=\eta\contr\alpha,\]
the corresponding interior product operator; here we adopt the convention that $\bi_{\alpha\wedge\beta}=\bi_{\alpha}\circ \bi_{\beta}$.
Moreover we shall denote by
$\bl_{\eta}=[\bi_{\eta},\de]\colon\Omega^{\ast}_{X}(U)\rightarrow\Omega^{\ast-i+1}_{X}(U)$ the holomorphic Lie derivative on differential forms.

The following algebraic structures play a major role in Poisson geometry:

\begin{definition} A Poisson algebra $(A,\cdot,\{,\})$ is a commutative algebra $(A,\cdot)$ equipped with a Lie bracket
$\{,\}$ verifying the Poisson identity:
\begin{eqnarray}\label{prel}\{a,b\cdot c\}=\{a,b\}\cdot c+b\cdot\{a,c\}\end{eqnarray}
A Gerstenhaber algebra $(A,\cdot,[\cdot,\cdot])$ is a graded commutative algebra $(A,\cdot)$ equipped with an odd bracket $[\cdot,\cdot]:A^{i}\otimes A^{j}\rightarrow A^{i+j-1}$ making $(A[1],[\cdot,\cdot])$ into a graded Lie algebra and verifying the odd Poisson identity:
\begin{eqnarray}\label{oddprel} [a,b\cdot c]=[a,b]\cdot c+ (-1)^{(|a|-1)|b|}b\cdot[a,c]\;.\end{eqnarray}
A differential Gerstenhaber algebra $(A,\cdot,[\cdot,\cdot],d)$ (DGeA for short) is a Gerstenhaber algebra $(A,\cdot,[\cdot,\cdot])$ equipped with a degree 1 differential $d$ making $(A,\cdot,d)$ into a DGA and $(A[1],[\cdot,\cdot],d)$ into a DGLA.
\end{definition}

\begin{remark}\label{remarkideali} Let $(A,\cdot,[\cdot,\cdot])$ be a  Gerstenhaber algebra
and let $I\subset A$ be the ideal of the underlying commutative graded algebra $(A,\cdot)$  generated  by a set of homogeneous elements $S\subset I$. Then $I$ is $[\cdot,\cdot]$-closed if and only if $[S,S]\subset I$.
\end{remark}

As a main example $\bigwedge^{\ast}\Theta_{X}$ carries a natural structure of sheaf of Gerstenhaber algebras with the exterior product and the Schouten-Nijenhuis bracket $[\cdot,\cdot]_{SN}$, see e.g. \cite{vaisman}.
Recall that
\[[\cdot,\cdot]_{SN}:{\bigwedge}^{i}\Theta_{X}\otimes{\bigwedge}^{j}\Theta_{X}\rightarrow{\bigwedge}^{i+j-1}\Theta_{X}\]
is defined uniquely according to \eqref{oddprel} so that  $[\eta,\xi]_{SN}$ is the usual bracket for
$\eta,\xi\in\Theta_{X}$, while for $\eta\in\Theta_{X}$ and $f\in\bigwedge^{0}\Theta_{X}=\Oh_{X}$ we have $[\eta,f]_{SN}=\eta(f)=\eta\contr df$:
the only thing to be checked to see this turns $\bigwedge^{\ast}\Theta_{X}$ into a sheaf of Gerstenhaber algebras is the (odd) Jacobi identity for $[\cdot,\cdot]_{SN}$, for this one reduce himself to check it on $\bigwedge^{\leq 1}\Theta_{X}$, which is plain.
Notice that for $\eta\in\Theta_{X}$ the operator $[\eta,\cdot]_{SN}$
is the Lie derivative with respect to $\eta$.

\begin{definition} A holomorphic Poisson bivector on $X$ is a section $\pi\in H^{0}(X,\bigwedge^{2}\Theta_{X})$ satisfying the integrability condition:
\begin{eqnarray}\label{Poissonbivector} [\pi,\pi]_{SN}=0\;.\end{eqnarray}
A holomorphic Poisson manifold is a pair $(X,\pi)$ consisting of a complex manifold $X$ and a holomorphic Poisson bivector $\pi$ on $X$.
\end{definition}

For instance, if $\dim X=2$ then every global  section of $\bigwedge^{2}\Theta_{X}$ is a  holomorphic Poisson bivector.
If $\dim X=3$, via the natural identification
\[ {\bigwedge}^{2}\Theta_{X}=\Omega^1_X\otimes {\bigwedge}^{3}\Theta_{X}=\Omega^1_X(K_X^{-1}),\]
a global section $\alpha\in H^0(X,\Omega^1_X(K_X^{-1}))$ corresponds to a holomorphic Poisson bivector if and only if
\[ \alpha\wedge d\alpha=0\in H^0(X,\Omega^3_X(K_X^{-2}))=H^0(X,K_X^{-1}).\]
In particular, if $\dim X=3$ and $K_X^{-1}=L^{2}$, then for every pair of sections $\beta,\gamma\in H^0(X,L)$ the section $\alpha=\beta d\gamma-\gamma d\beta\in H^0(X,\Omega^1_X(K_X^{-1}))$ gives a holomorphic Poisson bivector.

As another example, if $A\subseteq H^0(X,\Theta_X)$ is an abelian Lie subalgebra, then every element in the image of $\bigwedge^2A\to H^{0}(X,\bigwedge^{2}\Theta_{X})$ is a Poisson bivector.

Every holomorphic symplectic manifold $(X,\omega)$ is a holomorphic Poisson manifold, where
the Poisson bivector $\pi$ is uniquely determined by the condition
\[ \bi_{\pi}(\bi_{\eta}(\omega)\wedge\alpha)=\bi_{\eta}(\alpha),\qquad \eta\in\Theta_X,\qquad \alpha\in\Omega^1_X.\]

The datum of a holomorphic Poisson bivector $\pi$ on $X$ induces several additional structures, cf. \cite{KS08,vaisman}:

1) the \emph{Lichnerowicz-Poisson differential} $d_{\pi}=[\pi,\cdot]_{SN}\colon\bigwedge^{\ast}\Theta_{X}\rightarrow\bigwedge^{\ast+1}\Theta_{X}$ 
inducing on  $\bigwedge^{\ast}\Theta_{X}$ the structure of sheaf of DGeAs.

2) the \emph{Poisson bracket}
$\{\cdot,\cdot\}_{\pi}\colon\Oh_{X}\bigwedge\Oh_{X}\rightarrow\Oh_{X}$ given by
\[\{f,g\}_{\pi}=[[\pi,f]_{SN},g]_{SN}=[d_{\pi}f,g]_{SN}=\bi_{\pi}(\de f\wedge\de g)\,.\]
This clearly satisfies condition \eqref{prel} and it is well known that condition \eqref{Poissonbivector} is equivalent to the Jacobi identity for $\{\cdot,\cdot\}_{\pi}$. Therefore a holomorphic Poisson manifold could be equivalently defined as a complex manifold $X$ together with a sheaf of Poisson algebras structure on $\Oh_{X}$ (cf. \cite{vaisman}). We note that in a system of local holomorphic  coordinates $z_{1},\ldots,z_{n}$  we can reconstruct the Poisson bivector from the Poisson bracket by the formula
\[\pi=\sum_{1\leq i<j\leq n}\pi_{ij}\frac{\de}{\de z_{i}}\wedge\frac{\de}{\de z_{j}},
\qquad\text{ where }\quad\pi_{ij}=-\{z_{i},z_{j}\}_{\pi}\,.\]

3) the \emph{Koszul bracket} $[\cdot,\cdot]_{\pi}\colon\Omega^{i}_{X}\otimes\Omega^{j}_{X}\rightarrow\Omega^{i+j-1}_{X}$, defined by the formula:
\begin{eqnarray}\label{koszul}  [\alpha,\beta]_{\pi}:= (-1)^{i}(\bl_{\pi}(\alpha\wedge\beta)-\bl_{\pi}(\alpha)\wedge\beta)-\alpha\wedge\bl_{\pi}(\beta),\qquad \alpha\in\Omega^i_X,\end{eqnarray}
inducing on  $(\Omega^{\ast}_{X},\de)$ the structure of a sheaf of DGeAs 
(this is shown for instance in \cite{koszul,FMpoisson}).

4) the \emph{anchor map} $\pi^{\#}\colon\Omega^{\ast}_{X}\rightarrow\bigwedge^{\ast}\Theta_{X}$:
this is the morphism of graded sheaves defined for $\alpha\in\Omega^{1}_{X}$ by the formula
\begin{eqnarray}\label{anchor}
\pi^{\#}(\alpha)(f)= \bi_{\pi}(\alpha\wedge\de f),\qquad f\in\Oh_{X},
\end{eqnarray}
and then uniquely extended to an $\Oh_{X}$-linear morphism of sheaves of graded algebras.

\begin{remark}\label{rem.ancoramorfismoDGLA} It is well known
that
\[\pi^{\#}\colon (\Omega^*_X,\wedge,[\cdot,\cdot]_{\pi},\de)\to ({\bigwedge}^*\Theta_X, \wedge, [\cdot,\cdot]_{SN}, d_{\pi})\]
is a morphism of sheaves of differential Gerstenhaber algebras. Very briefly,
in order to prove that $\pi^{\#}([\alpha,\beta]_{\pi})=[\pi^{\#}(\alpha),\pi^{\#}(\beta)]_{SN}$ it is sufficient to check it for $\alpha\in\Omega^{1}_{X}$, $\beta\in\Omega^{0}_{X}=\Oh_{X}$, where it follows from definition (cf. \eqref{anchor} and \eqref{koszul}), 
or $\beta\in\Omega^{1}_{X}$, which is shown for instance in \cite{koszul, vaisman}; the proof of $\pi^{\#}(\de\alpha)=d_{\pi}(\pi^{\#}(\alpha))$ follows again from the Formula
\eqref{anchor} for $\alpha=f\in\Omega^{0}_{X}=\Oh_{X}$; in general it is sufficient to check it on forms of the type
$\alpha=f \,\de g_{1}\wedge\cdots\wedge\de g_{k}$, with $f,g_{1},\ldots,g_{k}\in\Oh_{X}$, which is a straightforward 
direct inspection.
\end{remark}

\begin{remark}
For a differential form $\alpha\in \Omega^k_X(U)$ lets denote by
$\alpha\wedge\colon \Omega^*_X(U)\to \Omega^{*+k}_X(U)$ the left multiplication by $\alpha$.
It is not difficult to prove the  formula
\[ \bi_{\pi^{\#}(\alpha)}=\frac{[\bi_{\pi},\cdot\;]^k}{k!}(\alpha\wedge),\]
where $[\,\cdot,\cdot\,]$ is the graded commutator in the space of endomorphisms of $\Omega^*_X(U)$.
When the Poisson bivector is induced by a holomorphic symplectic form $\omega$
the anchor map $\pi^{\#}\colon \Omega^1_X\to \Theta_X$
is an isomorphism with inverse $\omega^{\flat}\colon \Theta_X\to \Omega^1_X$,
$\omega^{\flat}(\eta)=\bi_{\eta}(\omega)$.

\end{remark}

\bigskip
\section{Deformations of holomorphic Poisson manifolds}

Let $X$ be a complex manifold, recall that a deformation of $X$  over $A\in\Art_{\C}$ is a pull-back diagram of complex spaces:
\begin{eqnarray}\label{deformation}
\xymatrix{X\ar[d]\ar[r]^{i}&\sX\ar[d]^{p}\\ \operatorname{Spec}\,\C\ar[r]&\operatorname{Spec}\,A}
\end{eqnarray}
with $p$ a smooth morphism; equivalently, it is the data of a sheaf $\Oh_{\sX}$ of flat unitary $A$-algebras on $X$ together with a morphism of sheaves of $A$-algebras  $\Oh_{\sX}\rightarrow\Oh_{X}$ which is locally isomorphic to the natural projection
$\Oh_{X}\otimes A\rightarrow \Oh_{X}$.
Equivalences between deformations $\sX_{0}$ and $\sX_{1}$ of $X$ over $A$ are isomorphisms of sheaves
of $A$-algebras $\Oh_{\sX_{0}}\rightarrow\Oh_{\sX_{1}}$ over $\Oh_{X}$; the self equivalences of the trivial deformation $\Oh_{X}\otimes A$ form a group canonically isomorphic to $\exp(H^{0}(X;\Theta_{X})\otimes\mathfrak{m}_A)$, being the isomorphism the usual exponential of derivations.

Given $\sU=\{U_{i}\}_{i\in I}$ an open covering of $X$ by Stein open sets, every deformation $\sX$ trivializes globally over each $U_{i}$, so that it can be reconstructed up to isomorphism by gluing the trivial deformations $\Oh_{U_{i}}\otimes A\rightarrow\Oh_{U_{i}}$ along double intersections via a family of transition automorphisms $e^{\eta_{ij}}\in\exp(\Theta_{X}(U_{ij})\otimes\mathfrak{m}_A)$ satisfying the cocycle condition $e^{\eta_{ij}}e^{\eta_{jk}}=e^{\eta_{ik}}\in\exp(\Theta_{X}(U_{ijk})\otimes\mathfrak{m}_A)$ on triple intersections (equivalently $\eta_{ij}\bullet\eta_{jk}=\eta_{ik}$, where $\bullet$ is the Baker-Campbell-Hausdorff product in the nilpotent Lie algebra $\Theta_{X}(U_{ijk})\otimes\mathfrak{m}_A$).

Consider the sheaf of $\Oh_{\sX}$-modules $\Theta_{\sX/A}$ of $A$-linear derivations of $\Oh_{\sX}$;  
as in the previous section the natural
structure of sheaf of Lie algebras on $\Theta_{\sX/A}$ extends to a sheaf of Gerstenhaber algebras structure on  $\bigwedge^{\ast}\Theta_{\sX/A}=\bigoplus_{i\geq0}\bigwedge^{i}_{\Oh_{\sX}}\Theta_{\sX/A}[-i]$ via the Schouten-Nijenhuis bracket; notice that if $U\subset X$ is a Stein open subset, then $\bigwedge^{\ast}\Theta_{\sX/A}(U)\cong\bigwedge^{\ast}\Theta_{X}(U)\otimes A$ with the Gersthenaber algebra structure given by scalar extension with $A$. Let as before $\sU$ be a covering of $X$ by Stein open sets and $e^{\eta_{ij}}$ denote transition automorphisms for $\Oh_{\sX}$ over $\sU$ with $\eta_{ij}\in\Theta_{X}(U_{ij})\otimes\mathfrak{m}_{A}\subset\Theta_{\sX/A}(U_{ij})$. The
adjoint operators
\[\ad\,\eta_{ij}=[\eta_{ij},\cdot]_{SN}\colon{\bigwedge}^{\ast}\Theta_{\sX/A}(U_{ij})\rightarrow{\bigwedge}^{\ast}\Theta_{\sX/A}(U_{ij})\]
are degree zero nilpotent Gerstenhaber derivations and their exponentials
$e^{\ad\,\eta_{ij}}$ are the transition automorphisms over $\sU$ for the sheaf $\bigwedge^{\ast}\Theta_{\sX/A}$, in fact it suffices to check this for $f\in\Oh_{\sX}(U_{ij})=\bigwedge^{0}\Theta_{\sX/A}(U_{ij})$ where $e^{\ad\,\eta_{ij}}(f)=e^{\eta_{ij}}(f)$ and for $\xi\in\Theta_{\sX/A}(U_{ij})$ where $e^{\ad\,\eta_{ij}}(\xi)=e^{\eta_{ij}}\circ\xi\circ e^{-\eta_{ij}}$. In the same way an equivalence between deformations $\sX_{0}$ and $\sX_{1}$ of $X$ over $A$ induce an isomorphism $\bigwedge^{\ast}\Theta_{\sX_{0}/A}\rightarrow \bigwedge^{\ast}\Theta_{\sX_{1}/A}$: if the equivalence is locally given by the isomorphisms
\[\xymatrix{\Oh_{\sX_{0}}(U_{i})\cong\Oh_{X}(U_{i})\otimes A\ar[r]^-{e^{\eta_{i}}}&\Oh_{X}(U_{i})\otimes A\cong\Oh_{\sX_{1}}(U_{i})}\;,\]
then the induced isomorphism is locally described by the maps
\[\xymatrix{\bigwedge^{\ast}\Theta_{\sX_{0}/A}(U_{i})\cong\bigwedge^{\ast}\Theta_{X}(U_{i})\otimes A\ar[r]^-{e^{\ad\,\eta_{i}}}&\bigwedge^{\ast}\Theta_{X}(U_{i})\otimes A\cong\bigwedge^{\ast}\Theta_{\sX_{1}/A}(U_{i})}.\]

\begin{definition}\label{def.deformationPoisson}
A deformation of a holomorphic Poisson manifold $(X,\pi)$ over $A\in\Art_{\C}$ is the data of a deformation of $X$ as in \eqref{deformation}
$\xymatrix{X\ar[r]^-{i}&\sX\ar[r]^-{p}&\operatorname{Spec}\,A}$ and of a section  $\widetilde{\pi}\in H^{0}(X;\bigwedge^{2}\Theta_{\sX/A})$ such that $[\widetilde{\pi},\widetilde{\pi}]_{SN}=0$ and such that $\widetilde{\pi}$
restricts to $\pi$ under the natural projection $\bigwedge^{\ast}\Theta_{\sX/A}\rightarrow\bigwedge^{\ast}\Theta_{X}$ (locally isomorphic to $\bigwedge^{\ast}\Theta_{X}\otimes A\rightarrow\bigwedge^{\ast}\Theta_{X}$).

Given two deformations $(\sX_{0},\widetilde{\pi}_{0})$, $(\sX_{1},\widetilde{\pi}_{1})$ an equivalence between them is an equivalence between $\sX_{0}$ and $\sX_{1}$ such that the induced isomorphism $\bigwedge^{\ast}\Theta_{\sX_{0}/A}\rightarrow\bigwedge^{\ast}\Theta_{\sX_{1}/A}$ sends $\widetilde{\pi}_{0}$ to $\widetilde{\pi}_{1}$. Thus, every holomorphic Poisson manifolds $(X,\pi)$  defines a formal pointed 
groupoid \[\Del_{(X,\pi)}:\Art_{\C}\rightarrow\Grpd,\]
sending $A$ to the groupoid whose objects are deformations of $(X,\pi)$ over $A$ and whose arrows are equivalences between them; similarly it defines  a formal set $\Def_{(X,\pi)}:\Art_{\C}\rightarrow\Set:A\rightarrow\pi_{0}(\Del_{(X,\pi)}(A))$, sending $A$ to the set of equivalence classes of deformations of $(X,\pi)$ over $A$.
\end{definition}

\begin{remark} Equivalently a deformation of $(X,\pi)$ over $A$ could be defined as a sheaf $\Oh_{\sX}$ of flat Poisson $A$-algebras on $X$ (by that we mean flat $A$-algebras equipped with an $A$-bilinear Poisson bracket) and a sheaves of Poisson $A$-algebras morphism $\Oh_{\sX}\rightarrow\Oh_{X}$ locally isomorphic to the projection $\Oh_{X}\otimes A\rightarrow\Oh_{X}$ (with the Poisson structure on $\Oh_{X}\otimes A$ given by scalar extension).
\end{remark}

In order to exhibit a differential graded Lie algebra governing the above deformation problem we are going to apply descent of Deligne groupoids (Theorem~\ref{thm.hinichdescent}); to this end  we consider 
the sheaf $\bigwedge^{\geq1}\Theta_{X}[1]$ of sub DGLAs of $\bigwedge^{\ast}\Theta_{X}[1]$.
Fixing an open Stein  covering $\sU=\{U_{i}\}$, consider the non-negatively graded semicosimplicial DGLA
\[\xymatrix{\bigwedge^{\geq1}\Theta_{X}[1](\sU)_{\bullet}:\quad
\displaystyle\prod_{i}^{\phantom{i}}{\bigwedge}^{\geq1}\Theta_{X}[1](U_{i})\ar@<0.65ex>[r]\ar@<-0.65ex>[r]&
\displaystyle\prod_{i,j}^{\phantom{i}}{\bigwedge}^{\geq1}\Theta_{X}[1](U_{ij})\ar@<1.00ex>[r]\ar[r]\ar@<-1.00ex>[r]&
\displaystyle\prod_{i,j,k}^{\phantom{i}}{\bigwedge}^{\geq1}\Theta_{X}[1](U_{ijk})\cdots}\]
with the usual \v{C}ech face operators given by restriction. The main result of this section is:

\begin{theorem}\label{thm.pdef} The totalization
$\Tot(\bigwedge^{\geq1}\Theta_{X}[1](\sU)_{\bullet})$ governs the deformations of $(X,\pi)$; more precisely
there exists an equivalence of formal pointed groupoids:
\[\Del_{\Tot(\bigwedge^{\geq1}\Theta_{X}[1](\sU)_{\bullet})}\simeq\Del_{(X,\pi)}\;.\]
\end{theorem}

\begin{proof} Given $A\in\Art_{\C}$ and a deformation $(\sX,\widetilde{\pi})$ of $(X,\pi)$ over $A$, since
$\sX$ trivializes over each $U_{i}$ we have  $\bigwedge^{2}\Theta_{\sX/A}(U_{i})\cong\bigwedge^{2}\Theta_{X}(U_{i})\otimes A$ and then
 we can write $\widetilde{\pi}_{|U_{i}}=\pi_{|U_{i}}+\sigma_{i}$ with
$\sigma_{i}\in\bigwedge^{2}\Theta_{X}(U_{i})\otimes\mathfrak{m}_{A}$.
Then $[\widetilde{\pi}_{|U_{i}}, \widetilde{\pi}_{|U_{i}}]_{SN}=0$ is equivalent to $[\pi_{|U_{i}},\sigma_{i}]_{SN}+\frac{1}{2}[\sigma_{i},\sigma_{i}]_{SN}=0$, i.e., $\sigma_{i}$ is a solution to the Maurer-Cartan equation in the DGLA $\bigwedge^{\geq1}\Theta_{X}[1](U_{i})\otimes\mathfrak{m}_{A}$. If the $e^{\eta_{ij}}$ are transition automorphisms for $\sX$ over $\sU$ on double intersections
we have $e^{\ad\,\eta_{ij}}(\pi_{|U_{ij}}+\sigma_{j|U_{ij}})=\pi_{|U_{ij}}+\sigma_{i|U_{ij}}$,
i.e., $e^{\eta_{ij}}\ast\sigma_{j|U_{ij}}=\sigma_{i|U_{ij}}$ in the DGLA $\bigwedge^{\geq1}\Theta_{X}[1](U_{ij})\otimes\mathfrak{m}_{A}$.

On the other hand by Theorem~\ref{thm.hinichdescent} the groupoid $\Del_{\Tot(\bigwedge^{\geq1}\Theta_{X}[1](\sU)_{\bullet})}(A)$ is naturally equivalent to the totalization of the semicosimplicial groupoid:
\[\xymatrix{\displaystyle\prod_{i}^{\phantom{i}}\Del_{\bigwedge^{\geq1}\Theta_{X}[1](U_{i})}(A)\ar@<0.65ex>[r]\ar@<-0.65ex>[r]&
\displaystyle\prod_{i,j}^{\phantom{i}}\Del_{\bigwedge^{\geq1}\Theta_{X}[1](U_{ij})}(A)\ar@<1.00ex>[r]\ar[r]\ar@<-1.00ex>[r]&
\displaystyle\prod_{i,j,k}^{\phantom{i}}\Del_{\bigwedge^{\geq1}\Theta_{X}[1](U_{ijk})}(A)\;\cdots}\]
Objects of this groupoid  are precisely the collections
\[\sigma_{i}\in{\bigwedge}^{2}\Theta_{X}(U_{i})\otimes\mathfrak{m}_{A},\qquad
e^{\eta_{ij}}\in\exp(\Theta_{X}(U_{ij})\otimes\mathfrak{m}_{A}),\]
such that for every $i,j,k$:
\begin{enumerate}
\item every $\sigma_{i}$ is a solution to the Maurer-Cartan equation in
$\bigwedge^{\geq1}\Theta_{X}[1](U_{i})\otimes\mathfrak{m}_{A}$;

\item $e^{\eta_{ij}}\ast\sigma_{j|U_{ij}}=\sigma_{i|U_{ij}}$ in
$\bigwedge^{\geq1}\Theta_{X}[1](U_{ij})\otimes\mathfrak{m}_{A}$;

\item  $e^{\eta_{ij}} e^{\eta_{jk}}=e^{\eta_{ik}}$ in $\exp(\Theta_{X}(U_{ijk})\otimes\mathfrak{m}_{A})$.

\end{enumerate}
The last condition ensures that we can glue via the $e^{\eta_{ij}}$ the trivial deformations over each $U_{i}$
to a global deformation $\sX$ of $X$ over $A$,  the second ensures that the local sections $\widetilde{\pi}_{i}=\pi_{|U_{i}}+\sigma_{i}$
paste to a global section $\widetilde{\pi}$ of $\bigwedge^{2}\Theta_{\sX/A}$, restricting to $\pi$, and the first one ensures that
$[\widetilde{\pi},\widetilde{\pi}]_{SN}=0$, i.e., $(\sX,\widetilde{\pi})$ is a deformation of $(X,\pi)$ over $A$, canonically associated to the descent data. Conversely, by the previous  discussion, every deformation of $(X,\pi)$ over $A$ determines descent data from which it can be reconstructed up to isomorphism.

A similar argument shows that equivalences of descent data correspond to equivalences of the associated deformations, and the other way around, so that it is defined this way a fully faithful, essentially surjective, functor
$\Del_{\Tot(\bigwedge^{\geq1}\Theta_{X}[1](\sU)_{\bullet})}(A)\rightarrow\Del_{(X,\pi)}(A)$.
Naturality of this construction in $A$ is clear.
\end{proof}

\bigskip
\section{Coisotropic deformations}
\label{sec.coisotropicdeformations}

We turn our attention to an important class of submanifolds of a holomorphic Poisson manifold, namely the coisotropic ones. Recall that a multiplicative ideal $I$ of a Poisson algebra $(A,\cdot, \{,\})$ is called coisotropic if it is closed with respect to $\{\cdot,\cdot\}$.

\begin{definition}\label{def.coisotropicsubmanifold}
Let $(X,\pi)$ be a holomorphic Poisson manifold.
A  holomorphic closed submanifold $Z\subset X$ is called coisotropic if its ideal sheaf $\sI_Z$
is  coisotropic in $\Oh_X$.
\end{definition}

For a closed submanifold $Z$ of a complex manifold $X$ we denote by $\sN_{Z|X}$ 
the normal sheaf of $Z$ in $X$ and by  $\bigwedge^{\ast}\sN_{Z|X}:=\bigoplus_{i\geq0}\bigwedge^{i}_{\Oh_{Z}}\sN_{Z|X}[-i]$ its graded exterior algebra. By a little abuse of notation we also denote by  
$\bigwedge^{\ast}\sN_{Z|X}$ its the direct image under the inclusion $Z\subset X$. Thus 
there is a natural epimorphism $\bigwedge^{\ast}\Theta_{X}\rightarrow \bigwedge^{\ast}\sN_{Z|X}$ of sheaves of graded algebras on $X$; we denote by $\sL^{\ast}_{Z}$ its kernel,  noticing that 
$\sL^{0}_{Z}=\sI_Z$ and $\sL^{1}_{Z}=\Theta_X(-\log Z)$.

\begin{proposition}\label{prop.cois}
In the notation above, the kernel
$\sL^{\ast}_{Z}$ is a sheaf of sub Gerstenhaber algebras of $\bigwedge^{\ast}\Theta_{X}$.
Moreover the following conditions are equivalent:\begin{enumerate}
\item  $Z$ is coisotropic;
\item  $\pi\in H^{0}(X;\sL^{2}_{Z})$;
\item $d_{\pi}(\sL^{\ast}_{Z})\subseteq \sL^{\ast}_{Z}$, i.e.,  $\sL^{\ast}_{Z}\subset\bigwedge^{\ast}\Theta_{X}$ is a sheaf of sub DGeAs.
\end{enumerate}
\end{proposition}

\begin{proof} We firts prove that $\sL^*_Z$ is generated, as a multiplicative ideal of $\bigwedge^*\Theta_X$ by $\sL^0_Z$ and $\sL^1_Z$: choosing a system of holomorphic coordinates 
$z_{1},\ldots,z_{n}$ such that $Z=\{z_{1}=\cdots=z_{p}=0\}$, we have 
$\dfrac{\de}{\de z_i}\in \sL^1_Z$ if and only if $i>p$, while
a polyvector field
\[\xi=\sum_{1\leq i_{1}<\cdots< i_{k}\leq n }\xi_{i_{1}\cdots i_{k}}
\frac{\de}{\de z_{i_{1}}}\wedge\cdots\wedge\frac{\de}{\de z_{i_{k}}}\]
belongs to $\sL^*_Z$ if and only if $\xi_{i_{1}\cdots i_{k}}\in\sI_{Z}$ whenever $1\leq i_{1}<\cdots< i_{k}\leq p$.

The first part of the proposition amounts to show that $\sL^{\ast}_{Z}$ is $[\cdot,\cdot]_{SN}$-closed; by
Remark~\ref{remarkideali} 
this is equivalent to the fact that $\sL^{\leq1}_{Z}$ is $[\cdot,\cdot]_{SN}$-closed, 
which is clear since
$\Theta_{X}(-\log Z)\subset\Theta_{X}$ is by definition the sub Lie algebra of derivations sending 
$\sI_{Z}$ into itself.

If  $\pi\in H^0(X,\sL^{2}_{Z})$ and $f,g\in\sI_{Z}=\sL^{0}_{Z}$ then also 
$\{f,g\}_{\pi}=[[\pi,f]_{SN},g]_{SN}\in\sI_{Z}$, so $Z$ is coisotropic and  
$\sL^{\ast}_{Z}\subset\bigwedge^{\ast}\Theta_{X}$ is a sub DGeA. 
To prove the converse let's write in local coordinates  
\[\pi=\sum_{1\leq i<j\leq n}\pi_{ij}\frac{\de}{\de z_{i}}\wedge\frac{\de}{\de z_{j}},\qquad
\pi_{ij}=-\{z_{i},z_{j}\}_{\pi}.\]
If $Z=\{z_1=\cdots=z_p=0\}$ is coisotropic then  $z_{i},z_{j}\in\sI_{Z}$ for 
$1\leq i<j\leq p$ and then  $\pi_{ij}=-\{z_{i},z_{j}\}_{\pi}\in\sI_{Z}$ 
for $1\leq i<j\leq p$; this implies that  $\pi$ is a section of $\sL^2_{Z}$.
\end{proof}

We shall first study coisotropic deformations of $(X,Z,\pi)$, recall that for $A\in\Art_{\C}$ a deformation of $(X,Z)$ seen as a pair (complex manifold, complex submanifold) over $A$ can be described as a deformation $\sX$ of $X$ together with a sheaf $\sI_{\sZ}\subset\Oh_{\sX}$ of $A$-flat ideals and a morphism:
\[\xymatrix{\sI_{\sZ}\ar@{^{(}->}[r]\ar[d]&\Oh_{\sX}\ar[d]\\ \sI_{Z}\ar@{^{(}->}[r]&\Oh_{X}}\]
of pairs of sheaves of $A$-algebras locally isomorphic to:
\[\xymatrix{\sI_{Z}\otimes A\ar@{^{(}->}[r]\ar[d]&\Oh_{X}\otimes A\ar[d]\\ \sI_{Z}\ar@{^{(}->}[r]&\Oh_{X}}\]
Equivalences are isomorphisms of pairs over $\sI_{Z}\hookrightarrow\Oh_{X}$,  the group of self-equivalences of the trivial deformation is naturally isomorphic to $\exp(H^{0}(X;\Theta_{X}(-\log Z))\otimes\mathfrak{m}_A)$, where
$\Theta_{X}(-\log Z)=\{\eta\in\Theta_{X}\mid \eta(\sI_{Z})\subset\sI_{Z}\}$ 
is the sheaf of vector fields tangent everywhere to $Z$. Again given a covering $\sU$ of $X$ by Stein open sets and a deformation of $(X,Z)$ as above it trivializes over each $U_{i}$, so that it can be reconstructed up to isomorphism by the family of transition automorphisms $e^{\eta_{ij}}\in\exp(\Theta_{X}(-\log Z)(U_{ij})\otimes\mathfrak{m}_A)$ satisfying the cocycle condition on triple intersections.

\begin{definition}\label{def2}
Given a holomorphic Poisson manifold $(X,\pi)$ and a coisotropic submanifold $Z\subset X$, a coisotropic deformation of the triple $(X,Z,\pi)$ is a deformation $(\sX,\widetilde{\pi})$ of $(X,\pi)$ equipped with a sheaf of coisotropic ideals $\sI_{\sZ}\subset\Oh_{\sX}$ such that $\sI_{\sZ}\hookrightarrow\Oh_{\sX}$ is a deformation of the pair $(X,Z)$. Together with the obvious notion of isomorphism, the deformations over a fixed basis have a natural groupoid structure. We denote by
\[\Del^{co}_{(X,Z,\pi)}\colon\Art_{\C}\rightarrow\Grpd,\qquad \Def^{co}_{(X,Z,\pi)}\colon\Art_{\C}\rightarrow\Set\,,\]
the associated formal pointed groupoid and functor of Artin rings.
\end{definition}

We take $\sU$ as usual and proceed as in the proof of Theorem~\ref{thm.pdef}, this time considering the non negatively graded semicosimplicial DGLA $\sL^{\geq1}_{Z}[1](\sU)_{\bullet}$.

\begin{theorem} There is an equivalence of formal groupoids:
\[\Del_{\Tot(\sL^{\geq1}_{Z}[1](\sU)_{\bullet})}\simeq\Del^{co}_{(X,Z,\pi)}\;.\]
\end{theorem}

\begin{proof} This is proved in the same way as in Theorem~\ref{thm.pdef}: 
the descent data for the semicosimplicial groupoid:
\[\xymatrix{\displaystyle\prod_{i}^{\phantom{i}}\Del_{\sL^{\geq1}_{Z}[1](U_{i})}(A)\ar@<0.65ex>[r]\ar@<-0.65ex>[r]&
\displaystyle\prod_{i,j}^{\phantom{i}}\Del_{\sL^{\geq1}_{Z}[1](U_{ij})}(A)\ar@<1.00ex>[r]\ar[r]\ar@<-1.00ex>[r]&
\displaystyle\prod_{i,j.k}^{\phantom{i}}\Del_{\sL^{\geq1}_{Z}[1](U_{ijk})}(A)\;\cdots}\]
can be glued to a deformation $\sI_{\sZ}\hookrightarrow\Oh_{\sX}$ of $(X,Z)$ and a deformation $\widetilde{\pi}$ of $\pi$ and conversely every deformation determines descent data from which it can be reconstructed up to isomorphism. The only thing that has to be observed is that given a solution $\sigma_{i}$ to the Maurer-Cartan equation in $\bigwedge^{\geq1}\Theta_{X}(U_{i})\otimes\mathfrak{m}_{A}$, in order for $\sI_{Z}(U_{i})\otimes A\subset\Oh_{X}(U_{i})\otimes A$ to be a coisotropic ideal with respect to the bracket induced by $\widetilde{\pi}_{i}=\pi_{|U_{i}}+\sigma_{i}$, it is necessary and sufficient that
$\sigma_{i}\in\sL^{2}_{Z}(U_{i})\otimes\mathfrak{m}_A$: this is shown as in the proof of 
Proposition~\ref{prop.cois}.
\end{proof}

As an application of the above result we are able to give the analog of Kodaira's stability theorem for coisotropic submanifolds.

\begin{corollary}[Stability of coisotropic submanifolds]\label{cor.stabilitycoisotropic}
Let $(X,\pi)$ be a compact holomorphic Poisson manifold and let $Z$ be a coisotropic submanifold.
Consider the complex of sheaves
\[ {\bigwedge}^{\ge 1}\sN_{Z|X}[1]:\qquad
\sN_{Z|X}\xrightarrow{d_{\pi}}{\bigwedge}^2 \sN_{Z|X}\xrightarrow{d_{\pi}}\cdots\]
where ${\bigwedge}^i \sN_{Z|X}$ is considered in degree $i-1$.
Let $\sX\to (B,0)$ be a Poisson deformation of $(X,\pi)$ over a germ of complex space $(B,0)$.
If $\H^1(Z,{\bigwedge}^{\ge 1}\sN_{Z|X}[1])=0$ then,
after a possible shrinking of $B$, there exists a  family of coisotropic submanifolds
$\sZ\subset \sX$  which is smooth over $B$ and such that $\sZ_0=Z$.
\end{corollary}

\begin{proof} Following the same standard argument used in the proof of Theorem~8.1 of \cite{Horidhm3}, involving relative Douady space and
Artin's theorem on the solution of analytic equations, it is not restrictive to assume $B$ a fat point, i.e.,
$B=\Spec(A)$ for some $A\in\Art_{\C}$.
Thus the stability theorem is proved whenever we show that the natural transformation of functors of Artin rings
\[ \Def^{co}_{(X,Z,\pi)}\mapor{\eta} \Def_{(X,\pi)}\]
is smooth; fixing an open Stein covering $\sU=\{U_i\}$ of $X$,
the above natural transformation is induced by the inclusion of differential graded
Lie algebras
\[ \Tot(\sL^{\geq1}_{Z}[1](\sU)_{\bullet})\mapor{i}\Tot({\bigwedge}^{\geq1}\Theta_{X}[1](\sU)_{\bullet}).\]
According to standard smoothness criterion, see e.g. \cite{ManettiRendiconti}, the morphism
$\eta$ is smooth whenever $i$ is surjective on $H^1$ and injective on $H^2$.
By the definition of $\sL^{\geq1}_{Z}$ we have an exact sequence of complexes of coherent sheaves
\[ 0\to \sL^{\geq1}_{Z}[1]\to {\bigwedge}^{\geq1}\Theta_{X}[1]\to {\bigwedge}^{\ge 1}\sN_{Z|X}[1]\to 0,\]
and the conclusion follows from Example~\ref{ex.coomologiadelTotversusHypercohomology} and
hypercohomology long exact sequence.
\end{proof}

Finally recall that for a complex manifold $X$ and a complex submanifold $Z\subset X$ the local Hilbert functor $\Hilb_{Z|X}\colon\Art_{\C}\rightarrow\Set$ sends $A$ to the set of sheaves of $A$-flat ideals $\sI_{\sZ}\subset\Oh_{X}\otimes A$ such that $\sI_{\sZ}\otimes_{A}\C=\sI_{Z}$; in other terms, $\Hilb_{Z|X}$  
is the functor of formal embedded deformations of $Z$ in $X$. If $X$ is Stein then every $\sI_{\sZ}\hookrightarrow\Oh_{X}\otimes A$ is isomorphic as a pair to 
$\sI_{Z}\otimes A\hookrightarrow\Oh_{X}\otimes A$, i.e., there exists $\eta\in H^{0}(X;\Theta_{X})\otimes\mathfrak{m}_A$ for which $\sI_{\sZ}=e^{\eta}(\sI_{Z}\otimes A)$.

\begin{definition}\label{def.coisotropicHilb} 
Given  a holomorphic Poisson manifold $(X,\pi)$ and  a coisotropic submanifold $Z\subset X$,
the local coisotropic Hilbert functor of $X$ in $Z$ is the functor of Artin rings
$\Hilb^{co}_{Z|X}\colon\Art_{\C}\rightarrow\Set$ sending $A$ to the set of sheaves of $A$-flat coisotropic ideals $\sI_{\sZ}\subset\Oh_{X}\otimes A$ such that $\sI_{\sZ}\otimes_{A}\C=\sI_{Z}$.
\end{definition}

Let $\sK^{\geq1}_{Z}$ be the homotopy fiber  of the inclusion of sheaves of (non negatively graded) DGLAs $\sL^{\geq1}_{Z}[1]\hookrightarrow\bigwedge^{\geq1}\Theta_{X}[1]$ (the notation is a little ambiguous, this is not the non negatively graded part of the homotopy fiber $\sK^{\ast}_{Z}$ of the inclusion $\sL^{\ast}_{Z}[1]\hookrightarrow\bigwedge^{\ast}\Theta_{X}[1]$); for an open covering $\sU$ let 
$\sK^{\geq1}_{Z}(\sU)_{\bullet}$ the associated semicosimplicial DGLA.

\begin{theorem}\label{thm.deformazionicoisotrope} For every open Stein covering $\sU$ of $X$ 
there exists an equivalence of formal pointed groupoids:
\[\Del_{\Tot(\sK^{\geq1}_{Z}(\sU)_{\bullet})}\simeq\Hilb^{co}_{Z|X}\,,\]
where $\Hilb^{co}_{Z|X}$ is regarded  via the natural inclusion
$\Set\rightarrow\Grpd$.
\end{theorem}

\begin{proof} We shall first show that $\Del_{\sK^{\geq1}_{Z}(U)}\simeq\Hilb^{co}_{U\bigcap Z|U}$ for a Stein open $U\subset X$, then the theorem will follow from descent of Deligne groupoids as in the previous cases. We saw in Example~\ref{ex.deligneperfibraootopica} that the groupoid $\Del_{\sK^{\geq1}_{Z}(U)}(A)$ admits the following description: objects are $e^{\eta}\in\exp(\Theta_{X}(U)\otimes\mathfrak{m}_A)$ such that $e^{-\eta}\ast0\in\sL^{2}_{Z}(U)\otimes\mathfrak{m}_A$, morphisms between objects $e^{\eta}$, $e^{\xi}$ are $e^{\alpha}\in\exp(\Theta_{X}(-\log Z)(U)\otimes\mathfrak{m}_A)$ such that $e^{\eta}=e^{\xi}e^{\alpha}$; moreover, the natural transformation $\Del_{\Tot(\sK^{\geq1}_{Z}(\sU)_{\bullet})}(A)\rightarrow\Def_{\Tot(\sK^{\geq1}_{Z}(\sU)_{\bullet})}(A)$ is an equivalence of groupoids.

The equivalence $\Del_{\sK^{\geq1}_{Z}(U)}(A)\simeq\Hilb^{co}_{U\bigcap Z|U}(A)$ 
is given on the set of objects sending $e^{\eta}$ to the ideal $e^{\eta}(\sI_{Z\cap U}\otimes A)\subset\Oh_{U}\otimes A$. 
This takes values in $\Hilb^{co}_{U\bigcap Z|U}(A)$, 
in fact applying the Gerstenhaber automorphism $e^{-\ad\,\eta}$ we see  that 
$e^{\eta}(\sI_{Z\cap U}\otimes A)$ 
is coisotropic if and only if $\sI_{Z\cap U}\otimes A$ is coisotropic 
with respect to the bracket induced by $e^{-\ad\,\eta}(\pi_{|U})=\pi_{|U}+e^{-\eta}\ast0$ 
and as in the proof of Proposition~\ref{prop.cois} this is equivalent to 
$e^{-\eta}\ast0\in\sL^{2}_{Z\cap U}\otimes\mathfrak{m}_A$; 
it is a morphism of groupoids as it factors through $\Del_{\Tot(\sK^{\geq1}_{Z}(\sU)_{\bullet})}(A)\rightarrow\Def_{\Tot(\sK^{\geq1}_{Z}(\sU)_{\bullet})}(A)$, since this is in equivalence it remains to show bijectivity of the induced $\Def_{\sK^{\geq1}_{Z}(U)}(A)\rightarrow\Hilb^{co}_{U\bigcap Z|U}(A)$: injectivity is plain, surjectivity follows from $U$ being Stein.

Observe it follows from the definitions that $U\rightarrow\Hilb^{co}_{U\bigcap Z|U}(A)$ is a sheaf of (pointed) sets on $X$, in particular this means that $\Hilb^{co}_{Z|X}(A)$ is in canonical bijective correspondence with the totalization of the semicosimplicial set:
\[\xymatrix{\prod_{i}\Hilb^{co}_{U_{i}\bigcap Z|U_{i}}(A)\ar@<0.65ex>[r]\ar@<-0.65ex>[r]&
\prod_{i,j}\Hilb^{co}_{U_{ij}\bigcap Z|U_{ij}}(A)\ar@<1.00ex>[r]\ar[r]\ar@<-1.00ex>[r]&\prod_{i,j,k}\Hilb^{co}_{U_{ijk}\bigcap Z|U_{ijk}}(A)\cdots}\]
and the first part of the proof gives a natural equivalence between this and the semicosimplicial groupoid:
\[\xymatrix{\prod_{i}\Del_{\sK^{\geq1}_{Z}(U_{i})}(A)\ar@<0.65ex>[r]\ar@<-0.65ex>[r]&
\prod_{i,j}\Del_{\sK^{\geq1}_{Z}(U_{ij})}(A)\ar@<1.00ex>[r]\ar[r]\ar@<-1.00ex>[r]&\prod_{i,j,k}\Del_{\sK^{\geq1}_{Z}(U_{ijk})}(A)\cdots}\]
thus a natural equivalence of the corresponding totalizations. Now Theorem~\ref{thm.hinichdescent} gives the desired natural equivalence $\Del_{\Tot(\sK^{\geq1}_{Z}(\sU)_{\bullet})}(A)\simeq\Hilb^{co}_{Z|X}(A)$.\end{proof}

\bigskip
\section{Coisotropic deformations induced by the anchor map}
\label{sec.anchordeformations}

In order to show that the coisotropic deformations induced by the anchor map are unobstructed we need an algebraic criterion for the homotopy abelianity of certain homotopy fibers, which we think of independent interest.

According to (generalized) Quillen's construction \cite[Prop.~3.3.2]{hinichDGC}, 
two DGLAs are quasi-isomorphic if and only if
they are weak equivalent as $L_{\infty}$-algebras; in particular a differential graded Lie algebra is homotopy abelian if and only if every bracket on its $L_{\infty}$ minimal model \cite{K} vanish.

Let $(L,d,[,])$ be a differential graded Lie algebra and assume that there exists an $L_{\infty}$-morphism
\[ f_{\infty}\colon (L,d,[,])\dashrightarrow (L,d,0),\quad f_\infty=\{f_n\},\quad f_n\colon L^{\wedge n}\to L,\]
with linear part $f_1$ equal to the identity, then $(L,d,[,])$ is homotopy abelian.
Notice that, according to homotopy classification of $L_{\infty}$-algebras \cite{K},
the converse is also true.

\begin{theorem}\label{thm.criterion}
Let $L=(L,d,[,])$ be a differential graded Lie algebra over a field of characteristic $0$.
If there exists a bilinear map $h\colon L\times L\to L$ of degree $-1$ such that
\begin{enumerate}
\item $h(a,b)=-(-1)^{\bar{a}\;\bar{b}}h(b,a)$,
\item\label{item.abelianitycriterion2} $[a,b]=dh(a,b)+h(da,b)+(-1)^{\bar{a}}h(a,db)$,
\item\label{item.abelianitycriterion3} $\displaystyle\oint [h(a,b),c]+\oint h([a,b],c)=0$, where $\oint$ denotes the sum over the cyclic permutations, taking care of Koszul signs.
\end{enumerate}
Then $L$ is homotopy abelian. If $M\subset L$ is a differential graded Lie subalgebra such that
$h(M\times M)\subset M$, then both $M$ and the homotopy fiber of the inclusion $M\subset L$ are homotopy abelian.
\end{theorem}

\begin{proof} (cf. \cite[Theorem~3.2]{FMpoisson}) According to the definition of $\oint$, for every multilinear map
$f\colon L\times L\times L\to L$  we have
\[ \oint f(a,b,c)=f(a,b,c)+(-1)^{\bar{a}(\bar{b}+\bar{c})}f(b,c,a)
+(-1)^{\bar{c}(\bar{a}+\bar{b})}f(c,a,b).\]
For instance, using this formalism, the graded Jacobi identity in a differential graded Lie algebra becomes
\[ \oint [[a,b],c]=0.\]
It is convenient to consider the $L_{\infty}[1]$-algebra $(V,q_1,q_2,0,\ldots)$ corresponding to $L$ via standard
d\'ecalage isomorphism:
\[V=L[1],\qquad q_1=-d,\qquad q_2(a,b)=(-1)^{\bar{a}}[a,b],\]
where $\bar{a}$ denotes the degree of $a$ in $L$.
The ``homotopy'' $h$ corresponds to a map
\[ r\in \Hom^0_{\K}(V^{\odot 2},V),\qquad r(a,b)=(-1)^{\bar{a}}h(a,b),\]
and it is straightforward to check that the above conditions
\eqref{item.abelianitycriterion2} and \eqref{item.abelianitycriterion3} become
\[ [r,q_1]_{NR}=q_2,\qquad [r,q_2]_{NR}=0,\]
where $[,]_{NR}$ is  the Nijenhuis-Richardson bracket on
$\Hom^*_{\K}(\bar{S^c}{V},V)$, i.e., the trasfer of the graded commutator bracket on the Lie algebra of coderivations of the reduced symmetric coalgebra $\bar{S^c}(V)$ via the corestriction isomorphism
$\Hom^*_{\K}(\bar{S^c}{V},V)\simeq \Coder^*_{\K}(\bar{S^c}{V},\bar{S^c}{V})$.

Let's denote by $R,Q_1,Q_2\in\Coder^*_{\K}(\bar{S^c}(V), \bar{S^c}(V))$ the coderivations with corestriction
$r,q_1,q_2$ respectively, we have that $R(V^{\odot n})\subset V^{\odot n-1}$ and then it is well defined an isomorphism of graded coalgebras $e^R\colon \bar{S^c}(V)\to\bar{S^c}(V)$. To conclude the proof it is sufficient
to prove that
\[ Q_1+Q_2=e^R\circ Q_1\circ e^{-R}=e^{[R,-]}(Q_1).\]
Taking corestrictions the above equality becomes
\[  q_1+q_2=e^{[r,-]_{NR}}(q_1)=q_1+[r,q_1]_{NR}+\frac{1}{2}[r,[r,q_1]]_{NR}+\cdots\]
which is clearly equivalent to $[r,q_1]_{NR}=q_2$ and $[r,q_2]_{NR}=0$.
If $P$ is a differential graded commutative algebra we can extend  the operator $h$ to
$P\otimes L$ in the obvious way:
\[  h(u\otimes a,v\otimes b)=(-1)^{\bar{u}+\bar{v}+\bar{v}\;\bar{a}}\; uv\otimes h(a,b)\;.\]
The validity of the  properties \eqref{item.abelianitycriterion2} and \eqref{item.abelianitycriterion3} for the operator $h$ in $P\otimes L$ is clear.\par

In particular $h$ extends to $L[t,dt]$ and commutes with the evaluation maps
$L[t,dt]\xrightarrow{t\mapsto 0}L$ and  $L[t,dt]\xrightarrow{t\mapsto 1}L$. Therefore $h$ extends to the homotopy fiber of the inclusion.
It is worth to notice here that, in general, the homotopy fiber of a morphism of 
homotopy abelian DGLAs is not homotopy abelian.
\end{proof}

\begin{corollary}\label{cor.criterionTOT}
Let
 \[
L_{\bullet}:\qquad \xymatrix{ {{L}_0}
\ar@<2pt>[r]\ar@<-2pt>[r] & { {L}_1}
      \ar@<4pt>[r] \ar[r] \ar@<-4pt>[r] & { {L}_2}
\ar@<6pt>[r] \ar@<2pt>[r] \ar@<-2pt>[r] \ar@<-6pt>[r]&
\cdots}
\]
be a semicosimplicial differential graded Lie algebra. Assume it is given a sequence of bilinear map
$h_n\colon L_n\times L_n\to L_n$ of degree $-1$ such that:
\begin{enumerate}

\item for every $n$ the map $h_n$ satisfies the conditions of Theorem~\ref{thm.criterion};

\item the sequence $\{h_n\}$ gives a semicosimplicial map, i.e., for every $n$ and every
face operator $\de_k\colon L_n\to L_{n+1}$ we have $h_{n+1}(\de_k a,\de_k b)=\de_k h_n(a,b)$.
\end{enumerate}
Then the totalization $\Tot(L_{\bullet})$ is a homotopy abelian differential graded Lie algebra.
\end{corollary}

\begin{proof} By scalar extension, every map $h_n$ extends to a bilinear map on $\Omega_n\otimes L_n$ and then
we have a bilinear map
\[ h\colon (\prod_{n}\Omega_n\otimes L_n)\times (\prod_{n}\Omega_n\otimes L_n)\to \prod_{n}\Omega_n\otimes L_n\;,\]
\[ h((a_1,a_2,\ldots),(b_1,b_2,\ldots))=(h_1(a_1,b_1), h_2(a_2,b_2),\ldots),\]
which satisfies the condition of Theorem~\ref{thm.criterion}. Since the morphisms $h_n$ commute with face operators, the totalization $\Tot(L_{\bullet})\subset \prod_{n}\Omega_n\otimes L_n$ is stable under $h$.
\end{proof}

\begin{lemma}\label{lem.dedfinizionedih} Let $(X,\pi)$ be a holomorphic Poisson manifold. 
Then the bilinear morphism 
\[h:\Omega^{i}_{X}\otimes\Omega^{j}_{X}\rightarrow\Omega^{i+j-2}_{X},\qquad h(\alpha,\beta)=(-1)^{i}(\bi_{\pi}(\alpha\wedge\beta)-\bi_{\pi}(\alpha)\wedge\beta-\alpha\wedge\bi_{\pi}(\beta)),\]
defined on the sheaf of DGLAs $(\Omega^{\ast}_{X}[1],[\cdot,\cdot]_{\pi},\de)$ verify the assumptions of Theorem~\ref{thm.criterion}. Given an open covering $\sU$ of $X$, there is induced a sequence of linear maps on the semicosimplicial DGLA $\Omega^{\ast}_{X}[1](\sU)_{\bullet}$ as in the hypotheses of Corollary~\ref{cor.criterionTOT}.
\end{lemma}

\begin{proof} The first part is a direct straightforward 
verification, substantially made in \cite{FMpoisson}; the last part is clear.
\end{proof}

For any integer $k\ge 0$, let $\Omega^{\geq k}_{X}$ denote the sheaf of sub DGeAs of $\Omega^{\ast}_{X}$ concentrated in degrees $\geq k$, in particular $\Omega^{\geq k}_{X}=\Omega^{\ast}_{X}$ for $k=0$. 
Notice that for $k\neq1$ the sheaf of DGLAs $\Omega^{\geq k}_{X}[1]$ is stable with respect to the operator  $h$ defined  in the previous Lemma~\ref{lem.dedfinizionedih}; this and Theorem~\ref{thm.criterion} immediately imply the following proposition.

\begin{proposition} 
Let $(X,\pi)$ be a holomorphic Poisson manifold. Then, for every nonnegative integer $k\neq1$,
$(\Omega^{\geq k}_{X}[1],[\cdot,\cdot]_{\pi},\de)$ is a sheaf of homotopy abelian DGLAs on $X$.
Moreover, for every open covering $\sU$ of $X$ the totalization 
$\Tot(\Omega^{\geq k}_{X}[1](\sU)_{\bullet})$ is homotopy abelian.
\end{proposition}

Next for a closed submanifold $Z\subset X$, let $\sI_{Z}^{\ast}\subset\Omega_{X}^{\ast}$ denote  
the ideal of forms vanishing along $Z$, i.e., the kernel of the natural restriction morphism 
$\Omega_{X}^{\ast}\rightarrow\Omega_{Z}^{\ast}$, also recall the sheaf 
$\sL^{\ast}_{Z}\subset\bigwedge^{\ast}\Theta_{X}$ defined in the previous section.

\begin{proposition}\label{prop.coisotripichstability}
Let $(X,\pi)$ be a holomorphic Poisson manifold and $Z\subset X$ a coisotropic submanifold.
Then $\sI_{Z}^{\ast}\subset\Omega_{X}^{\ast}$ is a sheaf of sub DGeAs; it
is closed with respect to the operator $h$ introduced in Lemma~\ref{lem.dedfinizionedih}
and $\pi^{\#}(\sI_{Z}^{\ast})\subset\sL^{\ast}_{Z}$. Conversely each one of these conditions is equivalent to $Z$ being coisotropic.
\end{proposition}

\begin{proof} Choose  local holomorphic  coordinates $z_{1},\ldots,z_{n}$ such that $Z=\{z_1=\cdots=z_p=0\}$.
If we assume that either $\sI_{Z}^{\ast}$ is $[\cdot,\cdot]_{\pi}$-closed, or that it is $h$-closed, or $\pi^{\#}(\sI_{Z}^{\ast})\subset\sL^{\ast}_{Z}$, for $1\leq i<j\leq p$ then
\[-\pi_{ij}=\bi_{\pi}(\de z_{i}\wedge \de z_{j})=-h(\de z_{i},\de z_{j})=\pi^{\#}(\de z_{i})(z_{j})=[\de z_{i},z_{j}]_{\pi}\in\sI_{Z}\]
and then  $Z$ is coisotropic (cf. the proof of \ref{prop.cois}). 

Since $\sI_{Z}^{\ast}\subset\Omega_{X}^{\ast}$ is the multiplicative ideal generated by $S=\{z_{j}$, $\de z_{i}\}_{1\leq i,j\leq p}$ the above computation also shows  the converse  
except $h$-closeness; in order to prove that 
$\bi_{\pi}(\alpha\wedge\beta)\in\sI_{Z}^{\ast}$ for $\alpha,\beta\in\sI_{Z}^{\ast}$ 
 and it is not restrictive to take as $\alpha$ an element of $S$. 
If $\alpha=z_{i}$ with $1\leq j\leq p$ the claim follows by $\Oh_{X}$-linearity of $\bi_{\pi}$, while for $\alpha=\de z_i$ the statement is equivalent to the fact that $\sI_{Z}^{\ast}$ is $(\bi_{\pi}\circ\de z_{j}\wedge)$-closed for $1\leq j\leq p$: this is true if and only if it is closed with respect to $[\bi_{\pi},\de z_{j}\wedge]=[\bi_{\pi},[\de, z_{j}\wedge]]=[\bl_{\pi},z_{j}\wedge]=[z_{j},\cdot]_{\pi}$, using formula \eqref{koszul}, which we already know.

Notice that for a coisotropic submanifold $Z$ the subspace  
$\sI_{Z}^{\ast}$ is  not $\bi_{\pi}$-closed in general.

\end{proof}

Denote with $\sJ^{\geq k}_{Z}$ the homotopy fiber of the inclusion of sheaves of DGLAs $\sI_{Z}^{\geq k}[1]\hookrightarrow\Omega_{X}^{\geq k}[1]$ (and with $\sJ^{\ast}_{Z}=\sJ^{\geq 0}_{Z}$).

\begin{proposition}\label{prop.abelianityoftots}
Let $(X,\pi)$ be holomorphic Poisson manifold, $Z\subset X$ a coisotropic submanifold and $k$ a nonnegative integer. If  $k\neq1$, then $\sI_{Z}^{\geq k}[1]$, $\sJ^{\geq k}_{Z}$ are sheaves of homotopy abelian DGLAs on $X$. Moreover, for an open covering $\sU$ of $X$ the totalizations $\Tot(\sI_{Z}^{\geq k}[1](\sU)_{\bullet})$, $\Tot(\sJ^{\geq k}_{Z}(\sU)_{\bullet})$ are homotopy abelian.
\end{proposition}

\begin{proof} Immediate from
Proposition~\ref{prop.coisotripichstability} and Theorem~\ref{thm.criterion}.
\end{proof}

In the same assumption of  Proposition~\ref{prop.abelianityoftots} we have a commutative diagram of DGLAs:

\begin{equation}
\xymatrix{\Tot(\sJ^{\ast}_{Z}(\sU)_{\bullet})\ar[r] & \Tot(\sI_{Z}^{\ast}[1](\sU)_{\bullet})\ar[r] & \Tot(\Omega^{\ast}_{X}[1](\sU)_{\bullet})\\
             \Tot(\sJ^{\geq 1}_{Z}(\sU)_{\bullet})\ar[r]\ar[u]& \Tot(\sI_{Z}^{\geq1}[1](\sU)_{\bullet})\ar[r]\ar[u]& \Tot(\Omega^{\geq 1}_{X}[1](\sU)_{\bullet})\ar[u]}
\end{equation}
with the DGLAs in the top row homotopy abelian and the inclusions as vertical arrows.

\begin{lemma} Let $\sU$ be an open Stein covering of $X$.
\begin{enumerate}

\item
if the Hodge to de Rham spectral sequence of $X$ degenerate at $E_1$, then the differential graded Lie algebra
$\Tot(\Omega^{\geq 1}_{X}[1](\sU)_{\bullet})$ is homotopy abelian.

\item if the Hodge to de Rham spectral sequence of $Z$ degenerate at $E_1$, then the differential graded Lie algebra
$\Tot(\sJ^{\geq 1}_{Z}(\sU)_{\bullet})$ is homotopy abelian.

\end{enumerate}
\end{lemma}

\begin{proof}
Recall that the Hodge to de Rham spectral sequence of a smooth complex manifold $X$ may be defined as the spectral sequence associated to the  filtration of \v{C}ech (double) complexes
$F^k=C(\sU,\Omega^{\ge k}_X)$ (see e.g. \cite{DI}).
Since the Whitney maps
\[ I\colon \Tot(\Omega^{\geq k}_{X}(\sU)_{\bullet})\to
C(\sU,\Omega^{\ge k}_X)\]
are quasi-isomorphisms of complexes, the first item is an immediate consequence of Proposition~\ref{prop.injective in cohomology}.

The second item  proved in the same way,  pointing out that for every Stein open subset $U\subset X$ and every $k\ge 0$ the complexes $\sJ^{\ge k}_Z(U)$ and $\Omega^{\ge k}_Z(U)$ are quasi-isomorphic.
\end{proof}

Next, fix an open Stein covering $\sU$ and consider the following commutative diagram of anchor maps: according to Remark~\ref{rem.ancoramorfismoDGLA} these are morphisms of differential graded Lie algebras.
\begin{equation}
\xymatrix{
\Tot(\sJ^{\geq 1}_{Z}(\sU)_{\bullet})\ar[r]\ar[d]^-{\pi^{\#}} & \Tot(\sI_{Z}^{\geq1}[1](\sU)_{\bullet})\ar[r]\ar[d]^-{\pi^{\#}} & \Tot(\Omega^{\geq 1}_{X}[1](\sU)_{\bullet})\ar[d]^-{\pi^{\#}}\\
             \Tot(\sK^{\geq1}_{Z}(\sU)_{\bullet})\ar[r] & \Tot(\sL_{Z}^{\geq1}[1](\sU)_{\bullet})\ar[r] & \Tot(\bigwedge^{\geq1}\Theta_{X}[1](\sU)_{\bullet})
}
\end{equation}

In particular the anchor map induces morphisms of deformation functors
\[ \pi^{\#}\colon \Def_{\Tot(\sJ^{\geq 1}_{Z}(\sU)_{\bullet})}\to \Hilb^{co}_{Z|X}\,,\qquad
\pi^{\#}\colon \Def_{\Tot(\Omega^{\geq 1}_{X}[1](\sU)_{\bullet})}\to \Def_{(X,\pi)}\,,\]
which at first order reduce to the anchor map in cohomology:
\[  \pi^{\#}\colon H^1(\Tot(\sJ^{\geq 1}_{Z}(\sU)_{\bullet}))=
H^1(Z,\Omega^{\ge 1}_Z)\to T^1\Hilb^{co}_{Z|X},\]
\[ \pi^{\#}\colon H^1(\Tot(\Omega^{\geq 1}_{X}[1](\sU)_{\bullet}))=
H^2(X,\Omega^{\ge 1}_X)
\to T^1\Def_{(X,\pi)}.\]

Whenever the Hodge to de Rham spectral sequence of $Z$ (resp.: $X$) degenerates at $E_1$
we have an isomorphism $H^1(Z,\Omega^{\ge 1}_Z)\simeq H^0(Z,\Omega^{1}_Z)$ (resp.:
$H^2(X,\Omega^{\ge 1}_X)\simeq H^0(X,\Omega^2_X)\oplus H^1(X,\Omega^1_X)$).

\begin{theorem}\label{thm.defcoisotropic}
In the notation above, if  the Hodge to de Rham spectral sequence of $Z$  degenerates at $E_1$,  
then  for every $\omega\in H^0(Z,\Omega^1_Z)$ the first order embedded coisotropic deformation 
$\pi^{\#}(\omega)$ extends to an embedded coisotropic deformation of $Z$
over $\Spec(\mathbb{C}[[t]])$.
\end{theorem}

\begin{proof} Clear, since for every open Stein covering $\sU$ the DGLA
$\Tot(\sJ^{\geq 1}_{Z}(\sU)_{\bullet})$ is homotopy abelian and
then the functor $\Def_{\Tot(\sJ^{\geq 1}_{Z}(\sU)_{\bullet})}$ is unobstructed.
\end{proof}

\begin{theorem}\label{thm.defpoisson}
In the notation above, if  the Hodge to de Rham spectral sequence of $X$  degenerates at $E_1$  then  for every
$\omega\in H^0(X,\Omega^2_X)\oplus H^1(X,\Omega^1_X)$ the first order deformation $\pi^{\#}(\omega)$ extends to an  deformation of $(X,\pi)$
over $\Spec(\mathbb{C}[[t]])$.
\end{theorem}

\begin{proof} As above,  for every open Stein covering $\sU$ the DGLA $\Tot(\Omega^{\geq 1}_{X}[1](\sU)_{\bullet})$ is homotopy abelian.
\end{proof}

The above theorem \ref{thm.defpoisson} has been proved is a different way by Hitchin \cite{hitchin} under the additional assumption that $X$ is compact K\"{a}hler. 
As a further application  we can generalize to coisotropic submanifolds part of
classical results by McLean and Voisin about deformation of
Lagrangian submanifolds \cite{mclean,voisin}.

\begin{corollary}\label{cor.unobstructness}
Let $Z$ be a compact coisotropic submanifold of a holomorphic Poisson manifold $(X,\pi)$.
Assume that the Hodge to de Rham spectral sequence of $Z$  degenerates at $E_1$ and the anchor map
\[ \pi^{\#}\colon H^0(Z,\Omega^1_Z)\to H^0(Z,\sN_{Z|X})\]
is surjective.
Then every small embedded deformation of $Z$ is coisotropic and the Hilbert functor
$\Hilb_{Z|X}=\Hilb^{co}_{Z|X}$ is
unobstructed.
\end{corollary}

\begin{proof} Since $Z$ is compact, by the argument used in
Corollary~\ref{cor.stabilitycoisotropic} it is sufficient to
consider infinitesimal deformations. It is now sufficient to apply Theorem~\ref{thm.defcoisotropic}.
\end{proof}

Obviously the above corollary fails without the assumption about the anchor
map. For instance, if  $Z=p$ is a point, then $Z$ is coisotropic if and only $\pi$ vanishes
at $p$; this shows that in general $\Hilb^{co}_{Z|X}$ is obstructed and strictly contained
in  $\Hilb_{Z|X}$. Corollary~\ref{cor.unobstructness} holds in particular for Lagrangian submanifolds of a holomorphic symplectic manifold;  a different proof of this case, 
based on Ran-Kawamata's $T^1$-lifting theorem, is given in \cite{KS13}.

\bigskip
\section{Dolbeault resolutions}
\label{sec.dolbeault}

In the previous sections we described the differential graded Lie algebras
controlling Poisson deformations and embedded coisotropic deformations using some purely
algebraic constructions, namely the Thom-Whitney-Sullivan totalization. Therefore all the above results can be easily extended to every algebraic Poisson manifold defined over a field of characteristic 0: 
roughly speaking, it is
sufficient to replace holomorphic with algebraic and Stein with affine and everything still works.

In this section we shall use Dolbeault's resolutions in order to give
another description of the DGLA governing embedded
coisotropic deformations; clearly this new
interpretation only work in the complex analytic setting.

Given a locally free sheaf $\sE$ on a complex manifold $X$ we shall denote by
$\sA_X^{0,j}(\sE)$ the sheaf of differentiable forms of type $(0,j)$ with values in $\sE$; the Dolbeault differential $\debar_{\sE}\colon \sA_X^{0,j}(\sE)\to \sA_X^{0,j+1}(\sE)$ is defined as the linear extension of
\[ \debar_{\sE}(\phi\otimes e)=\debar\phi\otimes e,\qquad \phi\in \sA_X^{0,j},\quad e\in \sE.\]

The Dolbeault resolution of  a bounded below  complex
\[(\sE^*,\delta)\colon\qquad  0\to \sE^i\mapor{\delta}\sE^{i+1}\mapor{\delta}\cdots\]
of locally free sheaves on a complex manifold is the  sheaf  of DG-vector spaces
$\sA_X^{0,*}(\sE^*)$, where
\[ \sA_X^{0,*}(\sE)^i=\bigoplus_{j+h=i}\sA_X^{0,j}(\sE^h),\]
and the differential $\debar_{\sE^*}$ is defined by the formula
\[ \debar_{\sE^*}\colon \sA_X^{0,j}(\sE^h)\to \sA_X^{0,j+1}(\sE^h)\oplus \sA_X^{0,j}(\sE^{h+1}),\qquad
\debar_{\sE^*}(\phi\otimes e)=\debar\phi\otimes e+(-1)^j\phi\otimes \delta e\;.\]
According to Dolbeault's lemma, the natural inclusion
$\sE^*\to \sA_X^{0,*}(\sE^*)$ is a quasi-isomorphism.

Similarly we  denote by $A_X^{0,*}(\sE^*)$
the DG-vector space  of global sections of the Dolbeault resolution; more generally, for every open subset $U\subset X$ we shall denote by $A_U^{0,*}(\sE^*)$ the DG-vector space of sections 
of $\sA_X^{0,*}(\sE^*)$ over $U$. Notice that, by Dolbeault theorem,  
the cohomology 
$A_X^{0,*}(\sE^*)$ is isomorphic to the 
hypercohomology of $\sE^*$.

Let $(\sE^*,\delta)$ be a bounded below complex of locally free sheaves on a complex
manifold $X$ and let $\sU=\{U_i\}$ be an open Stein covering of $X$. Thus we have a natural morphism of
semicosimplicial DG-vector spaces:
\[ \xymatrix{\sE^*(\sU)_{\bullet}:\ar[d]&
\prod_{i}\sE^*(U_i)\ar@<0.65ex>[r]\ar@<-0.65ex>[r]\ar[d]&
\prod_{i,j}\sE^*(U_{ij})\ar[d]\ar@<1.00ex>[r]\ar[r]\ar@<-1.00ex>[r]&\prod_{i,j,k}\sE^*(U_{ijk})\cdots\ar[d]\\
A_{\sU}^{0,*}(\sE^*)_{\bullet}:&
\prod_{i}A_{U_i}^{0,*}(\sE^*)\ar@<0.65ex>[r]\ar@<-0.65ex>[r]&
\prod_{i,j}A_{U_{ij}}^{0,*}(\sE^*)\ar@<1.00ex>[r]\ar[r]\ar@<-1.00ex>[r]&
\prod_{i,j,k}A_{U_{ijk}}^{0,*}(\sE^*)\cdots&}
\]
Since $A_X^{0,*}(\sE^*)$ is the equalizer of $\de_0,\de_1\colon A_{\sU}^{0,*}(\sE^*)_{0}\to
A_{\sU}^{0,*}(\sE^*)_{1}$ and
every map
\[ \sE^*(U_{i_1\cdots i_k})\to A_{U_{i_1\cdots i_k}}^{0,*}(\sE^*)\]
is a quasi-isomorphism, according to Remark~\ref{rem.equalizzatorecosimpliciale} 
there exists a diagram of quasi-isomorphisms
\[\xymatrix{&\Tot(\sE^*(\sU)_{\bullet})\ar[d]\\
A_X^{0,*}(\sE^*)\ar[r]^{e\quad}&\Tot(A_{\sU}^{0,*}(\sE^*)_{\bullet})},\qquad
e(x)=(1\otimes x,1\otimes \de_0 x,1\otimes \de_0^2 x,\ldots).\]

Here we apply the above general construction in two particular cases, both of them  
related with a coisotropic submanifold $Z$ of a  holomorphic Poisson manifold
$(X,\pi)$. 
In the first case we consider the complex of locally free sheaves on $X$
\[ {\bigwedge}^{\ge 1}\Theta_X[1]:\qquad 0\to \Theta_X\xrightarrow{d_{\pi}}{\bigwedge}^2\Theta_X\xrightarrow{d_{\pi}}
{\bigwedge}^3\Theta_X\cdots\;,\]
while in the second we consider the complex of locally free sheaves on $Z$
\[ {\bigwedge}^{\ge 1}\sN_{Z|X}[1]:\qquad 0\to
\sN_{Z|X}\xrightarrow{d_{\pi}}{\bigwedge}^2\sN_{Z|X}\xrightarrow{d_{\pi}}{\bigwedge}^3\sN_{Z|X}\cdots\;.\]
The complex $A_X^{0,*}({\bigwedge}^{\ge 1}\Theta_X[1])$ admits
a natural structure of differential graded Lie algebras, where the bracket is the
antiholomorphic extension of the Schouten-Nijenhuis bracket on
$\sA_X^{0,0}({\bigwedge}^{\ge 1}\Theta_X[1])$.

It is straightforward to check that there exists a short exact
sequence
\begin{eqnarray}\label{sequencealgebroid}
0\to L_{Z|X}\xrightarrow{\;\chi\;} A_X^{0,*}({\bigwedge}^{\ge 1}\Theta_X[1])
\xrightarrow{\;P\;}A_Z^{0,*}({\bigwedge}^{\ge 1}\sN_{Z|X}[1])\to 0,\end{eqnarray}
where $P$ is the natural projection map,  $L_{Z|X}$ is a differential graded subalgebra of
$A_X^{0,*}({\bigwedge}^{\ge 1}\Theta_X[1])$ and $\chi$ is the inclusion: in fact the proof of Proposition~\ref{prop.cois} implies that $L_{Z|X}\subset A_X^{0,*}({\bigwedge}^{\ge 1}\Theta_X[1])$ is a 
differential graded Lie subalgebra.

\begin{theorem}\label{dolbeaulthomfiber} In the notation of Section~\ref{sec.coisotropicdeformations}, for every open Stein covering $\sU$
of $X$,
the homotopy fiber of
the inclusion $L_{Z|X}\mapor{\chi} A_X^{0,*}({\bigwedge}^{\ge 1}\Theta_X[1])$
is quasi-isomorphic to
$\Tot(\sK^{\geq1}_{Z}(\sU)_{\bullet})$
and then governs the functor of infinitesimal embedded coisotropic deformations of $Z$ in $X$.
\end{theorem}

\begin{proof} Since $\Tot$ is an exact functor we have a short exact sequence
\[ 0\to \Tot(\sL_{Z}^{\geq1}[1](\sU)_{\bullet})\mapor{\alpha}\Tot({\bigwedge}^{\geq1}\Theta_{X}[1](\sU)_{\bullet})
\to \Tot({\bigwedge}^{\geq1}\sN_{Z|X}[1](\sU)_{\bullet})\to 0\]
and  $\Tot(\sK^{\geq1}_{Z}(\sU)_{\bullet})$ is isomorphic to the homotopy fiber of $\alpha$.
The above sequence is part of a $3\times 3$  diagram with exact rows,
with the first two columns  made by morphisms of differential graded Lie algebras and
where every vertical map   is a quasi-isomorphism:
\[ \xymatrix{0\ar[r]&\Tot(\sL_{Z}^{\geq1}[1](\sU)_{\bullet})\ar[r]^{\alpha}\ar[d]&
\Tot({\bigwedge}^{\geq1}\Theta_{X}[1](\sU)_{\bullet})\ar[r]\ar[d]&
\Tot({\bigwedge}^{\geq1}\sN_{Z|X}[1](\sU)_{\bullet})\ar[r]\ar[d]&0\\
0\ar[r]&K\ar[r]&
\Tot(A_{\sU}^{0,*}({\bigwedge}^{\geq1}\Theta_{X}[1])_{\bullet})\ar[r]&
\Tot(A_{\sU}^{0,*}({\bigwedge}^{\geq1}\sN_{Z|X}[1])_{\bullet})\ar[r]&0\\
0\ar[r]&L_{Z|X}\ar[r]^{\chi}\ar[u]& A_X^{0,*}({\bigwedge}^{\ge 1}\Theta_X[1])\ar[r]\ar[u]_e&
A_Z^{0,*}({\bigwedge}^{\ge 1}\sN_{Z|X}[1])\ar[r]\ar[u]_e& 0}\]
This diagram induces a quasi-isomorphism between the homotopy fibers of $\alpha$ and $\chi$.
\end{proof}

\bigskip
\section{Relation with the homotopy Lie algebroid}

Coisotropic deformations have been studied, in the differentiable setting, by
using $L_{\infty}$-algebras together with
Voronov's construction of higher derived brackets
\cite{CF07,CS08,FZ12,OP05,voronov}.

\begin{theorem}[Th. Voronov \cite{voronov2}]\label{thm.Voronov}
Let $(M, [\cdot,\cdot])$ be a graded Lie algebra, splitting, as a graded vector space, in the direct sum  $M=L\oplus A$,
where $L,A$ are graded  Lie subalgebras of $M$ and $A$ is abelian;  denote by $P\colon M\rightarrow A$ the projection with kernel $L$. For every derivation 
$D\in\Der^*(M)$ such that $D(L)\subset L$, its higher derived bracket are defined as
\begin{eqnarray}\label{derivedbrackets}\qquad
\{\cdots\}^n_D\colon A^{\odot n}\rightarrow A,\qquad \{a_{1},\ldots,a_{n}\}^n_D=
P[[\cdots[[Da_1,a_2],a_3],\ldots\;],a_n]\,,\qquad n\geq1,
\end{eqnarray}

Then, for every degree one derivation $D\in\Der^1(M)$ such that $D^2=0$ and $D(L)\subset L$,
the higher derived brackets of $D$
give a structure of $L_{\infty}$-algebra on $A[-1]$.
\end{theorem}

Notice that $\{a\}^1_D=PDa$ and the 
graded symmetry follows easily from Leibniz rule and the abelianity of $A$: for instance
\[ \{a,b\}^2_D=P[Da,b]=P(D[a,b]-(-1)^{\bar{a}\bar{D}}[a,Db])=
(-1)^{\bar{a}(\bar{D}+\bar{Db})}P[Db,a]=(-1)^{\bar{a}\;\bar{b}}\{b,a\}^2_D\,.\]
If $D(A)\subset A$, then $\{a\}^1_D=a$ and $\{\cdots\}^n_D=0$ for every $n>1$.

The link between higher derived brackets and homotopy fibers is given by the 
following result: the reader can find a proof in the paper \cite{derived}.

\begin{theorem}\label{thm.derivedvshomfiber}
In the same setup of  the above Theorem~\ref{thm.Voronov}, if $D$ has degree 1 and $D^2=0$, then
the $L_{\infty}$-algebra $(A[-1],\{\cdot\}^1_D,\{\cdot,\cdot\}^2_D,\ldots)$ is weak equivalent to the homotopy fiber of the inclusion of DGLAs
\[(L,D,[\cdot,\cdot])\hookrightarrow(M,D,[\cdot,\cdot])\,.\] 
\end{theorem}

In order to apply the higher derived bracket construction to the study of holomorphic coisotropic deformations,
we look for a splitting $A_Z^{0,\ast}(\bigwedge^{\geq1} \sN_{Z|X}[1])\rightarrow A_X^{0,\ast}(\bigwedge^{\geq1} \Theta_{X}[1])$ of the exact sequence~\eqref{sequencealgebroid}, such that the image is an 
abelian graded Lie subalgebra of
$A_X^{0,\ast}(\bigwedge^{\geq1} \Theta_{X}[1])$. In the differentiable setting something similar is accomplished after restricting to a tubular neighborhood of $Z$ in $X$; in the complex analytic setting,
however,  one has to work from the outset in the rather restrictive hypothesis that $X=E$ is the total space of a holomorphic vector bundle
$p\colon E\rightarrow Z$ over $Z$, which is embedded in $E$ as the zero section.

Denoting by $N_{Z|E}$ the normal bundle, there is a canonical identification
$N_{Z|E}\cong E$, in this way the pull-back bundle $p^{\ast}N_{Z|E}\rightarrow E$ is canonically identified with the subbundle $p^{\ast}E\subset TE$ of vertical tangent vectors.
The above induces a morphism $\sN_{Z|E}\rightarrow p_{\ast}\Theta_{E}$ of sheaves on $Z$,
sending a section $\xi$ of $N_{Z|E}$ to the vector field constantly $\xi_{p}$ along the fiber $E_p$; by multiplicative extension we also get $\bigwedge^{\geq1}\sN_{Z|E}\rightarrow p_{\ast}\bigwedge^{\geq1}\Theta_{E}$.
Thus, for every open subset $U\subset Z$ we have a morphism
\[ {\bigwedge}^{\geq1}\sN_{Z|E}[1](U)\rightarrow {\bigwedge}^{\geq1}\Theta_{E}[1](p^{-1}(U))\]
whose image is an abelian graded Lie subalgebra.
Acting via  pull-back on differential forms, we obtain a splitting:
\begin{equation}\label{splitting}
\sigma\colon A_Z^{0,\ast}({\bigwedge}^{\geq1} \sN_{Z|E}[1])\xrightarrow{\qquad}
 A_E^{0,\ast}({\bigwedge}^{\geq1} \Theta_{E}[1])
\end{equation}
of the exact sequence \eqref{sequencealgebroid} with the required properties.

Suppose given a Poisson bivector $\pi$ on $E$ such that $Z\subset E$ is a coisotropic submanifold:
we look at $\pi$ as a section of
$A^{0,0}_E(\bigwedge^2\Theta_E)\subset \left(A^{0,\ast}_E(\bigwedge^{\geq1}\Theta_E[1])\right)^1$,
then it follows from Proposition \ref{prop.cois} that $\pi\in L^1_{Z|E}$.
Since $\pi$ is holomorphic Poisson, putting $D = \debar + d_{\pi}$ we are in the algebraic setup of
Theorem~\ref{thm.Voronov}.

Denote with $P\colon A^{0,\ast}_E(\bigwedge^{\geq1}\Theta_E[1])\rightarrow A^{0,\ast}_Z(\bigwedge^{\geq1}\sN_{Z|E}[1])$
the projection as in the exact sequence~\eqref{sequencealgebroid}, and take
$\sigma:A^{0,\ast}_Z(\bigwedge^{\geq1}\sN_{Z|E}[1])\rightarrow A^{0,\ast}_E(\bigwedge^{\geq1}\Theta_E[1])$
as in \eqref{splitting}. Since the image of $\sigma$ is $\debar$-closed,  the higher derived brackets
$\{\cdots\}^n_D\colon A_Z^{0,\ast}(\bigwedge^{\geq1} \sN_{Z|E}[1])^{\odot n}\rightarrow A_Z^{0,\ast}(\bigwedge^{\geq1} \sN_{Z|E}[1])$  are equal to 
\[ \{\xi\}^1_D=P(\debar\sigma(\xi)+d_{\pi}(\sigma(\xi)))=(\debar+d_{\pi})\xi,\]
\[\{\xi_1,\ldots,\xi_n\}^n_D=P([[\cdots[d_{\pi}(\sigma(\xi_1)),\sigma(\xi_2)]_{SN},\cdots],\sigma(\xi_n)]_{SN}),\qquad
n\ge 2,\]
and define an $L_{\infty}$ structure on  $A_Z^{0,\ast}(\bigwedge^{\geq1} \sN_{Z|E})$.

Thus, according to Theorem~\ref{thm.derivedvshomfiber},
the corresponding $L_{\infty}$ structure on $A_Z^{0,\ast}(\bigwedge^{\geq1} \sN_{Z|E})$
is weak equivalent to the homotopy fiber of
$\chi\colon L_{Z|X}\mapor{} A_X^{0,*}({\bigwedge}^{\ge 1}\Theta_X[1])$; therefore
Theorems~\ref{dolbeaulthomfiber} and \ref{thm.deformazionicoisotrope} immediately imply the following:

\begin{corollary} In the above hypotheses there is an isomorphism of functors of Artin rings
\[\Def_{A_Z^{0,\ast}(\bigwedge^{\geq1} \sN_{Z|E})}\cong\Hilb^{co}_{Z|E}\;.\]
\end{corollary}

\begin{remark} The $L_{\infty}$-algebra $A_Z^{0,\ast}(\bigwedge^{\geq1} \sN_{Z|E})$
is concentrated in degrees $\geq1$; in particular the elements of
$\Def_{A_Z^{0,\ast}(\bigwedge^{\geq1} \sN_{Z|E})}(A)$ correspond in a bijective way to solutions
$\xi\in A^{0,0}_Z(\sN_{Z|E})\otimes\mathfrak{m}_A$ of the Maurer-Cartan equation:
\[\sum_{n\geq1}\frac{\{\xi,\ldots,\xi\}^n_D}{n!}=0\,.\]
\end{remark}

\end{document}